\documentclass[12pt]{article}
\usepackage{amsfonts,amssymb,amsmath,amsthm}
\usepackage{mathrsfs}

\theoremstyle{plain}
\newtheorem{thm}                 {Theorem}     
\newtheorem{proposition}  [thm]  {Proposition}
\newtheorem{prop}  [thm]  {Proposition}
\newtheorem{corollary}    [thm]  {Corollary}
\newtheorem{lemma}        [thm]  {Lemma}

\theoremstyle{definition}

\newtheorem{remark}       [thm]  {Remark}
\newtheorem{defn}   [thm]  {Definition}
\newtheorem{conv}   [thm]  {Convention}

\newcommand{\ignore}[1]{}

\newcommand{\R}{{\mathbb{R}}}
\newcommand{\Z}{{\mathbb{Z}}}
\newcommand{\N}{{\mathbb{N}}}
\newcommand{\C}{{\mathbb{C}}}

\newcommand{\CP}{{\mathbb{C}}{{P}}}
\newcommand{\RP}{{\mathbb{R}{{P}}}}
\newcommand{\beq}{\begin{equation}}
\newcommand{\eeq}{\end{equation}}
\newcommand{\bea}{\begin{eqnarray}}
\newcommand{\eea}{\end{eqnarray}}
\newcommand{\ben}{\begin{eqnarray*}}
\newcommand{\een}{\end{eqnarray*}}
\newcommand{\ra}{\rightarrow}

\newcommand{\cd}{\partial}

\newcommand{\wt}{\widetilde}
\newcommand{\wh}{\widehat}

\newcommand{\ip}[1]{\langle#1\rangle}
\newcommand{\n}[1]{| \|#1\| |}

\def \d{\mathrm{d}}

\newcommand{\dist}{{\rm dist}\, }
\newcommand{\spec}{{\rm spec}\, }

\newcommand{\eps}{\varepsilon}

\newcommand{\M}{\mathsf{M}}
\newcommand{\XX}{\mathsf{X}}
\newcommand{\YY}{\mathsf{Y}}
\newcommand{\UU}{\mathsf{U}}

\newcommand{\hh}{\mathsf{H}}
\newcommand{\lip}{\mathsf{Lip}}

\newcommand{\ul}{\underline}

\renewcommand{\theequation}{\arabic{section}.\arabic{equation}}
\renewcommand{\thethm}{\arabic{section}.\arabic{thm}}
\newcommand{\news}{\setcounter{thm}{0} \setcounter{equation}{0}}

\begin{document}

\title{The adiabatic limit of wave map flow on a two torus}
\author{
J.M. Speight\thanks{E-mail: {\tt speight@maths.leeds.ac.uk}}\\
School of Mathematics, University of Leeds\\
Leeds LS2 9JT, England
}

\date{}
\maketitle

\begin{abstract}
The $S^2$ valued wave map flow on a Lorentzian domain $\R\times \Sigma$, where
$\Sigma$ is any flat two-torus, is studied. The Cauchy problem with initial
data tangent to the moduli space of holomorphic maps $\Sigma\ra S^2$ is
considered, in the limit of small initial velocity. It is proved that
wave maps, in this limit, converge in a precise sense to geodesics in the
moduli space of holomorphic maps, with respect to the $L^2$ metric.
This establishes, in a rigorous setting, a long-standing informal conjecture
of Ward.
\end{abstract}

\section{Introduction}\news

Wave maps are the analogue of harmonic maps in the case where the domain is
Lorentzian. They satisfy a semilinear wave equation which has been heavily
studied,
in the simplest nontrivial case of $S^2$ target space,
 as a model PDE system involving a manifold-valued field \cite{shastr}.
The wave map equation is particularly interesting in the case where the
domain is $(\R\times\Sigma,dt^2-g_\Sigma)$,
with $(\Sigma,g_\Sigma)$ an oriented Riemannian
two-manifold. 
In this case, the static wave map problem is conformally invariant, so 
static solutions on $\Sigma=\R^2$ in particular have no preferred
scale: they can be dilated to any size without changing their energy. This
suggests that time-dependent solutions with initial data close to a static
solution might collapse and form singularities in finite time,
an issue which has been heavily studied both numerically and analytically
mainly for $\Sigma=\R^2$, $N=S^2$, within a certain
rotational equivariance class.
Numerical studies of increasing sophistication suggested that finite-time 
collapse can occur, and suggested formal models of the collapse process
\cite{leepeyzak,bizchmtab,linsad,ovcsig}. The first rigorous proof of blow-up came in 
the work of Krieger, Schlag and Tataru \cite{krischtat}, who proved the
existence of rotationally equivariant initial data, of topological degree
$n=1$, leading to
finite-time collapse. Rodnianski and Sterbenz subsequently
proved existence of equivariant initial data of every degree $n\geq 4$
leading to finite-time collapse, and proved that collapse is stable
to small perturbations of the initial data, at least within the 
equivariance class \cite{rodste}. These results were extended to every
degree $n\geq 1$ in work of Raphael and Rodnianski \cite{raprod}, 
which also established
detailed asymptotics and universality properties of the collapse mechanism. For a
thorough discussion of blow-up of wave maps, see \cite{ovcsig,raprod}.

This paper addresses a different analytic issue from singularity formation,
namely the validity of the {\em geodesic approximation} to wave map flow.
To motivate this, one should think of wave maps $\R\times\Sigma\ra S^2$ as
(formal) critical points of the action functional
\beq\label{lico}
S=\int dt\, (T-E),\qquad\mbox{where}\quad T=\frac12\int_\Sigma|\phi_t|^2,\quad
E=\frac12\int_\Sigma|d_\Sigma\phi|^2.
\eeq
It follows (from Noether's Theorem) that they conserve the total
energy $T+E$. A rather general argument of
Lichnerowicz \cite{lic} shows that 
for any  map $\phi:\Sigma\ra S^2$ of topological degree $n\in\Z$
(subject to suitable boundary conditions, if $\Sigma$ is
noncompact), $E\geq 4\pi |n|$, with equality if and only is $\phi$ is
$\pm$ holomorphic. 
So holomorphic maps, if they exist, minimize potential energy in their
homotopy class. Let us denote by $\M_n$ the moduli space of degree $n$
holomorphic maps $\Sigma\ra S^2$.
 Consider a wave map $\phi(t)$ with $\phi(0)\in \M_n$ and
$\phi_t(0)\in T_{\phi(0)}\M_n$ with 
$\|\phi_t(0)\|_{L^2}$ small. By conservation of 
$E+\frac12\|\phi_t(t)\|_{L^2}^2$, one expects that $\phi(t)$ will stay close to
$\M_n$, on which $E$ attains its minimum value, for as long as the solution 
persists. 
This led Ward to suggest \cite{war}, 
in the specific case $\Sigma=\R^2$,
that such wave maps should be well approximated by 
the dynamical system with action $S$, but with $\phi(t)$ {\em constrained} to
$\M_n$ for all time.  Since $E$ is constant on $\M_n$, this constrained
system is
equivalent to geodesic motion on $\M_n$ with respect to the
$L^2$ metric (obtained by restricting the quadratic form $T$ to $T\M_n$).
A similar approximation had previously been proposed by Manton \cite{man}
for low energy monopole dynamics, and the geodesic approximation is now
a standard technique in the study of the dynamics of topological solitons
\cite{mansut}.

Geodesic motion on $\M_2$ (for $\Sigma=\R^2$) was studied in detail in
\cite{lee}. There is a technical problem: the $L^2$ metric is only well-defined
on the leaves of a foliation of $\M_n$ and one must impose by hand that
$\phi(t)$ remains on a single leaf. This turns out to be ill-justified (it
precludes singularity formation for $n=1$, for example, in contradiction 
of \cite{krischtat,raprod}). This technical deficiency is removed if we choose
$\Sigma$ to be a compact Riemann surface. Here geodesic motion in $\M_n$ is
globally well-defined, if incomplete \cite{sadspe}, and the $L^2$
geometry of $\M_n$ is quite well understood, at least for some choices
of $\Sigma$ and $n$ \cite{macspe,spe_jmp,spe_T2,spe_rat1}. 

The question remains: is geodesic motion in $\M_n$ really a good
approximation to wave map flow in the adiabatic (low velocity) limit?
The purpose of this paper is to prove that it {\em is}, 
for times of order $(\mbox{initial velocity})^{-1}$ at least in the
case where $\Sigma$ is any flat two-torus. More precisely, we will prove:

\begin{thm}[Main Theorem]
\label{lilcol}
 Let $\M_n$ denote the moduli space of 
degree $n\geq 2$ holomorphic maps from a flat two-torus $\Sigma$ to $S^2$.
For fixed $\phi_0\in \M_n$ and $\phi_1\in T_{\phi_0}\M_n$ consider the one 
parameter
family of initial value problems for the wave map equation with $\phi(0)=\phi_0$, 
$\phi_t(0)=\eps\phi_1$, parametrized by $\eps>0$. There exist constants $\tau_*>0$ 
and $\eps_*>0$, depending only on the initial data,
such that for all $\eps\in(0,\eps_*]$, the problem has a unique solution for 
$t\in [0,\tau_*/\eps]$. Furthermore, the time re-scaled solution
$$
\phi_\eps:[0,\tau_*]\times\Sigma\ra S^2,\qquad
\phi_\eps(\tau,p)=\phi(\tau/\eps,p)
$$
converges uniformly in $C^1$ to $\psi:[0,\tau_*]\times\Sigma\ra S^2$, the 
geodesic in $\M_n$ with the same initial data, as $\eps\ra 0$.
\end{thm}

To prove this we will adapt the perturbation method devised by Stuart to prove validity of 
the geodesic approximation in the critically coupled
abelian Higgs and Yang-Mills-Higgs models \cite{stu1,stu2}. The wave map
problem has a key similarity with these gauge-theoretic problems, namely
a moduli space of static solutions which minimize energy in their homotopy
class and satisfy a system of first order ``Bogomol'nyi'' equations.
(For wave maps, the Bogomol'nyi equation is the condition that $\phi$
be $\pm$ holomorphic, i.e.\ the Cauchy-Riemann equation.) Roughly, the idea
is to decompose the solution $\phi(t)$ as $\phi(t)=\psi(t)+\eps^2Y(t)$
where $\psi(t)\in\M_n$, and control the growth of a suitable Sobolev norm
of the error $Y(t)$ uniformly in $\eps$ by means of energy estimates.
One concurrently shows that the projected trajectory $\psi(t)$ converges to a 
geodesic in $\M_n$.

In comparison with Stuart's work on vortices and monopoles, the situation
we study is simpler in two respects: we work on a compact domain $\Sigma$
(rather than $\R^2$ or $\R^3$), and our system has no gauge symmetry.
On the other hand, the wave map problem introduces two new challenges
for the method.

First, our field is manifold-valued, so it is not clear a priori what the
decomposition $\phi(t)=\psi(t)+\eps^2 Y(t)$ really means. In preliminary work
on this problem, it was suggested that the correct formulation
was $\phi(t)=\exp_{\psi(t)}\eps^2Y(t)$, where $\exp:TS^2\ra S^2$ is the 
exponential map \cite{hasspe}. In fact, this turns out {\em not} to have
the analytic properties required by Stuart's method (except for rotationally
equivariant wave maps). In this paper we isometrically embed
$S^2$ in $\R^3$ and use the ambient linear structure to project as usual,
$\phi(t)=\psi(t)+\eps^2Y(t)$. This choice is simple, but has significant
repercussions: $Y$ is no longer tangent to the map $\psi$ 
(not a section of $\psi^{-1}TS^2$), and must satisfy a nonlinear pointwise 
constraint to ensure that $\phi$ is $S^2$ valued. The evolution of $Y$ is
governed by a nonlinear wave equation whose (spatial) linear part is the
Jacobi operator $J_\psi$ for the harmonic map $\psi:\Sigma\ra S^2$. It turns
out that $J_\psi$ is {\em not} self-adjoint when acting on non-tangent
sections (such as $Y$). Since self-adjointness of (the analogue of)
$J_\psi$ is crucial for Stuart's method, we must devise a way round this:
we replace $J_\psi$ by an ``improved'' Jacobi operator $L_\psi$, which coincides 
with $J_\psi$ on tangent sections, but is self-adjoint on all sections, and
introduce compensating nonlinear terms into the wave equation for $Y$ using
the pointwise constraint. Further difficulties result: $L_\psi$,
unlike $J_\psi$, does not define a coercive quadratic form on the $L^2$
orthogonal complement of $T_\psi \M_n$. We must work instead with a weaker
near-coercivity property, which turns out to suffice for our purposes.

Second, while $\Sigma$ is compact, the moduli space $\M_n$ is not.
Of course, the vortex and monopole moduli spaces, dealt with by Stuart, are
also noncompact, but in those cases, moving to infinity corresponds to
(clusters of) solitons separating off and escaping to infinite 
separation, a well-controlled process. For wave maps, by contrast, approaching
the boundary of $\M_n$ at infinity corresponds to one or more lumps
collapsing and ``bubbling off''. In this process, $\psi$ becomes singular and
both geodesic motion and wave map flow become badly behaved. To handle this, 
we must keep careful track of the position (of $\psi\in\M_n$) dependence of
our various estimates, and modify Stuart's a priori energy bound so that we
simultaneously control the error $Y(t)$, the deviation of $\psi(t)$ from the
corresponding geodesic, and the distance of $\psi(t)$ from $\cd_\infty\M_n$.

It is interesting to speculate to what extent Theorem \ref{lilcol} can be
generalized. It is clear that the proof presented here generalizes quite easily 
to the case of a general compact Riemann surface, provided $n$ is sufficiently large 
compared with the genus of $\Sigma$. The 
reason for restricting to the case $\Sigma=T^2$ is mainly one of presentation: 
the existence of global cartesian coordinates makes it straightforward
to define the various function spaces, for example. Generalizing the
target space is not so straightforward. The wave map flow $\R\times\Sigma\ra N$
has the appropriate ``Bogomol'nyi'' form for Stuart's method to apply
whenever $\Sigma,N$ are both compact k\"ahler manifolds (in fact, it suffices
for $\Sigma$ to be co-k\"ahler). The choice $N=\CP^k$, $k\geq 2$, is of some 
interest in mathematical physics, for example. But here the reliance on
an isometric embedding $N\subset\R^p$ becomes very problematic. It seems
likely that some variant of Theorem \ref{lilcol} does remain true for
general compact k\"ahler targets, but proving it would require a rather 
different approach, perhaps along the lines sketched in \cite{stu3}.

One should note that Theorem \ref{lilcol} gives no information about singularity
formation for wave maps on $\R\times\Sigma$ because, although there certainly
are geodesics $\psi(\tau)$ which hit $\cd_\infty\M_n$ in time $\tau_0<\infty$, and the corresponding
wave maps do converge uniformly to $\psi(\tau)$ on some interval $[0,\tau_*]$, there is
no reason to expect $\tau_*=\tau_0$. In fact Raphael and Rodnianski have shown that
singularity formation of equivariant wave maps on $\R^2$ deviates significantly from 
the dynamics predicted by the (suitably regulated) geodesic approximation \cite{raprod}. Since
blow up is a (spatially) local phenomenon, these results presumably apply in some form
on the torus, which would imply $\tau_*<\tau_0$. Nontheless, the geodesic approximation
(on compact $\Sigma$ or, regulated, on $\Sigma=\R^2$) predicted finite time blow-up
of wave maps in $(2+1)$ dimensions, and this prediction turned out to be correct. 
The geodesic approximation also makes predictions about the {\em genericity} of blow up.
It is not hard to prove, for example,  that generic geodesics on
$\M_1$ for $\Sigma=S^2$ do {\em not} hit the boundary at infinity. It would be interesting to see
whether the full wave map flow has this property (i.e.\ generic Cauchy data tangent to $\M_1$
have global smooth solutions). The analogue of Theorem \ref{lilcol} for $\Sigma=S^2$
could provide a starting point for proving such results.

The rest of the paper is structured as follows. In section \ref{sec:mod} the moduli space
$\M_n$ of holomorphic maps is introduced and its key property, Proposition \ref{psiprop}
(existence of a smooth
local parametrization about any point), established. In section \ref{sec:proj}, the 
projection of the wave map flow to $\M_n$ is defined, and the coupled system satisfied by
$\psi$ and the error section $Y$ is derived,
equation (\ref{cs}). In section \ref{sec:prelim}, some standard
functional analytic definitions and results are introduced. Our aim here, and in
the remainder of the
paper, is to make the proof accessible to a wide mathematical physics audience, not just
experts in PDE.\, In section \ref{sec:locex} a local existence and uniqueness theorem
for the coupled system (\ref{cs}) governing $(\psi,Y)$ is proved, Theorem \ref{locex}. 
Of course, local 
existence and uniqueness 
of {\em wave maps} in this context is not new; the extra, and new, information we
obtain here is local existence and
uniqueness of the {\em projection} to $\M_n$.  This is the engine underlying Stuart's method,
and we go through the argument in some detail, not only because there are certain new aspects we
have to deal with which Stuart did not (e.g.\ preservation of the pointwise constraint on $Y$),
but also because the requirements of Picard's method for this proof determine our choice of function
spaces, a point which is not obvious (to the non-analyst) in Stuart's original applications
of the method \cite{stu1,stu2}. In section \ref{sec:coerce} the key near-coercivity property
of the quadratic form associated with the improved Jacobi operator is proved,
Theorem \ref{Q2thm} (roughly, that
$\ip{L_\psi Y,L_\psi L_\psi Y}_{L^2}$ controls the $H^3$ norm of $Y$). In section \ref{energyestimates}
energy estimates are established which bound the growth of $Y(t)$. Finally, in section
\ref{sec:long}, the coercivity properties and energy estimates are combined to prove long time 
existence of the solution $(\psi,Y)$, and establish convergence to the corresponding
geodesic in $\M_n$. An appendix presents the proofs of some
 basic analytic properties of the nonlinear terms in the
coupled system for $(\psi,Y)$ which are used repeatedly in section \ref{sec:locex}.

\section{The moduli space of static $n$-lumps}\news
\label{sec:mod}

Static wave maps are harmonic maps $\phi:\Sigma\ra S^2$, that is, solutions
of the harmonic map equation
\beq
\phi_{xx}+\phi_{yy}+(|\phi_x|^2+|\phi_y|^2)\phi=0
\eeq
or, equivalently, critical points of the Dirichlet energy
\beq
E(\phi)=\frac12\int_\Sigma |\phi_x|^2+|\phi_y|^2.
\eeq
Maps $\Sigma\ra S^2$ fall into disjoint homotopy classes labelled by their
degree $n\in\Z$, which we may assume, without loss of generality, is 
non-negative. An argument of Lichnerowicz \cite[p39]{eellem} shows that,
in the degree $n$ class, $E(\phi)\geq 4\pi n$, with equality if and only if
$\phi$ is holomorphic. 
Furthermore, all harmonic maps $\Sigma\ra S^2$ of
degree $n\geq 2$ are holomorphic \cite{eelwoo}.
So the moduli space of interest, $\M_n$, is the
space of degree $n$ holomorphic maps $\Sigma\ra  S^2$. Such maps are called
``$n$-lumps'' by analogy with the case $\Sigma=\C$, where the Dirichlet
energy density typically exhibits $n$ distinct local maxima, which may 
loosely be thought of as smoothed out particles, or lumps of energy.

The global topology
of the space $\M_n$ is quite complicated, for example, $\M_2\cong [\Sigma\times PSL(2,\C)]/(\Z_2\times\Z_2)$,
\cite{spe_T2}. For our purposes local
information will suffice, however. Given a smooth variation  $\phi_s$ of 
$\phi=\phi_0\in\M_n$ we have $dE(\phi_s)/ds=0$ at $s=0$ (since $\phi$ is
harmonic) and
\beq
\frac{d^2E(\phi_s)}{ds^2}=\ip{V,J_\phi V}=\int_{\Sigma}V\cdot J_\phi V.
\eeq
where $V=\cd_s\phi_s|_{s=0}\in\Gamma(\phi^{-1}TS^2)$ is the infinitesimal
generator of the variation and
\beq\label{Jdef}
J_\phi V=-V_{xx}-V_{yy}-(|\phi_x|^2+|\phi_y|^2)V
-2(\phi_x\cdot V_x+\phi_y\cdot V_y)\phi
\eeq
is the Jacobi operator at the map $\phi$ \cite[p155]{ura}. This operator is
self-adjoint and elliptic, and its spectrum determines the stability properties
of $\phi$. By the Lichnerowicz argument, $\phi$ minimizes
$E$ in its homotopy class, so $\spec J_\phi$ is non-negative. 
Given a variation $\phi_s$ through {\em harmonic} maps, that is, a curve in
$\M_n$ through $\phi=\phi_0$, its infinitesimal generator
 $V=\cd_s\phi_s|_{s=0}$ satisfies $J_\phi V=0$. 
Hence $T_\phi\M_n\subset\ker J_\phi$.
For a general harmonic map $\phi:M\ra N$
between Riemannian manifolds, the converse may be false, that is, 
there may be sections
$V\in\ker J_\phi\subset \Gamma(\phi^{-1}TN)$ which are not tangent to
any variation of $\phi$ through harmonic maps, and in this case the space
of harmonic maps $M\ra N$ may not be a smooth manifold around $\phi$. It is
important for us to rule out this kind of bad behaviour in our case.
More precisely, we need:

\begin{prop}\label{psiprop}
Given any $\phi_0\in\M_n$, $n\geq 2$, there exists an open set
$U\subset\R^{4n}$ and a smooth map $\psi:U\times\Sigma\ra S^2$ such that,
\begin{itemize}
\item[(i)] for each $q\in U$, $\psi(q,\cdot)\in\M_n$,
\item[(ii)] there exists $q_0\in U$ such that $\phi_0=\psi(q_0,\cdot)$, and
\item[(iii)] $\psi_\mu=\cd\psi/\cd q^\mu$, $\mu=1,2,\ldots,4n$ 
span $\ker J_{\psi(q,\cdot)}$.
\end{itemize}
\end{prop}

\begin{proof}
Choose any $p\in S^2$ such that both $p$ and $-p$ are regular values of
$\phi_0$ (such $p$ exists by Sard's Theorem). Then $\phi:\Sigma\ra S^2$
is in $\M_n$ if and only if $s_p\circ\phi$,
its image under stereographic projection from
$p$, is meromorphic, of degree $n$, that is, a degree $n$ elliptic
function. The most general degree $n$ elliptic
function is \cite{law}
\beq
(s_p\circ\phi)(z)=\lambda\frac{\sigma(z-a_1)\cdots\sigma(z-a_n)}{\sigma(z-b_1)\cdots
\sigma(z-b_n)}
\eeq
where $\sigma$ is the Weierstrass sigma function, $\lambda,a_1,\ldots,
a_n,b_1,\ldots,b_n$ are complex constants, $\lambda\neq 0$,
$\sum a_i=\sum b_i\mod\Lambda$ and $\{a_i\}\cap\{b_j\}=\emptyset$.
Hence, we may parametrize a general point $\phi\in\M_n$ by $4n$ real numbers
$q^\mu$,
for example, the real and imaginary parts of $\lambda,
a_1,\ldots,b_{n-1}$ having set $b_n=a_1+\cdots+a_n-b_1-\cdots-b_{n-1}$.
Further, $\phi$ manifestly depends smoothly on $q$ and $z$. Hence
we have a smooth map $\psi:\wt{U}\times\Sigma\ra S^2$ satisfying properties
(i) and (ii). By our choice of $p$, $s_p\circ\phi_0$ has $n$ distinct
zeroes and $n$ distinct poles, so $\{\psi_\mu\}$ at $q=q_0$ are 
linearly independent
sections of $\psi(q,\cdot)^{-1}TS^2$, and hence, by smoothness, also linearly independent on
some neighbourhood $U$ of $q_0$ in $\wt{U}$. As explained previously,
$\psi_\mu$ span a subspace of $\ker J_{\psi(q,\cdot)}$,
so it remains to show that $\ker J_\phi$ has dimension $4n$ for any
$\phi\in \M_n$.

It is known \cite[p174]{ura}
that $\ker J_\phi$ is isomorphic, as a complex vector space,
 to $H^0(\Sigma, L)$, the space of holomorphic
sections of the line bundle $L=\phi^{-1}T'S^2$, where $T'S^2$ denotes the 
holomorphic tangent bundle of $S^2$. Since $\phi$ has degree $n$ and
$T'S^2$ has degree $2$, $L$ has degree $2n$. Now, by the Riemann-Roch
formula \cite{hitsegwar}
\beq
\dim H^0(\Sigma,L)-\dim H^1(\Sigma,L)=\deg L=2n
\eeq
since $\Sigma$ has genus $1$. But, by Serre duality, $H^1(\Sigma,L)\cong
H^0(\Sigma,K\otimes L^*)^*$ where $K$ is the canonical bundle of $\Sigma$.
Now $K$ is trivial, so $K\otimes L^*$ has degree $-2n$, and hence has no 
holomorphic sections, whence $H^1(\Sigma,L)=0$. It follows that
$\ker J_\phi$ has real dimension $4n$, as required.
\end{proof}

We can regard $q^\mu$ as local coordinates on $\M_n$. 
Given the initial data $\phi_1\in T_{\phi_0}\M_n$ of interest, we choose
and fix such a $\psi:U\times \Sigma\ra S^2$ and denote by $q_0\in U$ and
$q_1\in\R^{4n}$ those vectors such that $\phi_0=\psi(q_0,\cdot)$ and
$\phi_1=q_1^\mu\psi_\mu(q_0,\cdot)$. Here, as henceforth, we use the
Einstein summation convention on repeated indices. We also choose
and fix a compact neighbourhood $K$ of $q_0$ in $U$. In a slight abuse
of notation, we will also use the symbol $\psi$ to denote the
associated map $U\supset K\ra \M_n\subset (S^2)^\Sigma$, so $\psi(q)$ will
denote the holomorphic map $z\mapsto \psi(q,z)$. We will also use $\psi(t)$ as shorthand for
$\psi(q(t))$, meaning 
a general curve in $\M_n$.

There is a natural Riemannian metric on $\M_n$, the $L^2$ metric, whose
components in the local coordinate system $q^\mu$ are
\beq\label{metdef}
\gamma_{\mu\nu}(q)=\ip{\psi_\mu,\psi_\nu}=\int_\Sigma\frac{\cd\psi}{\cd q^\mu}\cdot\frac{\cd\psi}{\cd q^\nu}.
\eeq
The associated Christoffel symbol is
\beq\label{chrisdef}
G^\mu_{\lambda\nu}(q)=\gamma^{\mu\alpha}\ip{\psi_\alpha,\psi_{\lambda\nu}}
\eeq
where $\gamma^{\mu\nu}$ is the inverse metric and $\psi_{\mu\nu}=\cd^2\psi/\cd q^\mu\cd q^\nu$. This is
the metric whose geodesics approximate wave maps in the adiabatic limit.

\section{Projection of wave map flow and the coupled system}\news
\label{sec:proj}

The wave map equation for $\phi:\R\times\Sigma\ra S^2\subset\R^3$ is
\beq\label{wavemapeq}
\phi_{tt}-\phi_{xx}-\phi_{yy}+(|\phi_t|^2-|\phi_x|^2-|\phi_y|^2)\phi=0.
\eeq
Given $\eps>0$, a small parameter, we decompose $\phi$ as
\beq\label{proj}
\phi=\psi+\eps^2 Y
\eeq
where, at each fixed time, $\psi(t,\cdot):\Sigma\ra S^2$ is a degree $n$ 
harmonic map, and $Y:\Sigma\ra \R^3$.
We may think of $\psi(t)$ as a curve in $\M_n$, the moduli space of degree 
$n$ harmonic maps, and $Y$ as the
``error''
incurred by projecting $\phi(t)$ to $\psi(t)$. It is useful to think of $Y$ as a section of $\psi^{-1}\ul{\R^3}$, where $\ul{\R^3}=S^2\times\R^3$ is the trivial $\R^3$ bundle over $S^2$, and $\psi^{-1}\ul{\R^3}$
is its pullback to $\Sigma$. With this is mind, we refer to $Y$ as the ``error section'', and to any $Z:\Sigma\ra\R^3$ with $Z\cdot\psi=0$ pointwise as a "tangent section" (in bundle language, $Z$ is a section of $\psi^{-1}TS^2
\subset\psi^{-1}\ul{\R^3}$). Clearly, $Y$ is {\em not} a tangent section (unless $Y=0$). Since both $\psi$ and
$\phi$ are $S^2$ valued, $Y$ must satisfy the pointwise constraint
\beq\label{pc}
\psi\cdot Y=-\frac12\eps^2|Y|^2.
\eeq
For a given curve $\psi(t)$, if $\phi$ is a wave map then $Y$ must satisfy the PDE obtained by substituting
(\ref{proj}) into (\ref{wavemapeq}),
\beq\label{YPDE1}
Y_{tt}+J_\psi Y=k+\eps j
\eeq
where $J_\psi$ is the Jacobi operator associated with the harmonic map $\psi(t):\Sigma\ra S^2$ 
 and the terms on the
right hand side are
\bea
k&=&-(\psi_{\tau\tau}+|\psi_\tau|^2\psi)\nonumber\\
j&=&2(\psi_\tau\cdot Y_t)\psi+\eps\{(|Y_t|^2-|Y_x|^2-|Y_y|^2)\psi
+(|\psi_\tau|^2-2\psi_x\cdot Y_x-2\psi_y\cdot Y_y)Y\}\nonumber \\
&&+2\eps^2(\psi_\tau\cdot Y_t)Y+\eps^3(|Y_t|^2-|Y_x|^2-|Y_y|^2)Y.\label{jdef}
\eea
We have here introduced the slow time variable $\tau=\eps t$ as a book-keeping device. The precise expression for
$j$ is not important. What matters is its qualitative form: it depends only on $\psi,\psi_\tau$ and
$Y$ and its first derivatives, and the dependence is smooth (polynomial, in fact).

Superficially (\ref{YPDE1}) looks exactly analogous to the corresponding equation in Stuart's analysis of 
vortex dynamics \cite{stu1}, but this is deceptive. As already noted, the Jacobi operator is a self-adjoint, elliptic,
second order linear operator $J_\psi:\Gamma(\psi^{-1}TS^2)\ra \Gamma(\psi^{-1}TS^2)$ whose spectrum determines the stability
properties of the harmonic map $\psi$ \cite{ura}. It is important
to realize, however, that in (\ref{YPDE1}) $J_\psi$ is acting on $Y$, which is {\em not} a tangent section. So
in (\ref{YPDE1}), $J_\psi$ is precisely the same operator defined above (\ref{Jdef}), but extended to
act on sections of $\psi^{-1}\ul{\R^3}$. But $J_\psi$ is {\em not} self-adjoint (with respect to $L^2$)
as an operator on $\psi^{-1}\ul{\R^3}$, and self-adjointness of (the analogue of) $J_\psi$ is a crucial 
ingredient in Stuart's method. In fact
\beq
J_\psi Z=-\Delta Z-(|\psi_x|^2+|\psi_y|^2)Z+A_\psi Z
\eeq
where the non-self-adjoint piece and its adjoint are
\bea
 A_\psi Z&=&-2(\psi_x\cdot Z_x+\psi_y\cdot Z_y)\psi \nonumber\\
A_\psi^\dagger Z&=&-2\left\{(\psi\cdot Z)\Delta \psi+(\psi\cdot Z)_x\psi_x+(\psi\cdot Z)_y\psi_y\right\}.
\eea
and we have adopted the analysts' convention for the Laplacian, that is, $\Delta Z=Z_{xx}+Z_{yy}$. 
To remedy this deficiency, we make the following definition:

\begin{defn} Given a harmonic map $\psi:\Sigma\ra S^2$, we define its {\em improved Jacobi operator} to
be
$$
L_\psi:\Gamma(\psi^{-1}\ul{\R^3})\ra\Gamma(\psi^{-1}\ul{\R^3}),\qquad
L_\psi=J_\psi+A_\psi^\dagger.
$$
Note that $L_\psi$ coincides with $J_\psi$ on $\Gamma(\psi^{-1}TS^2)$, and hence $L_\psi$ maps tangent sections
to tangent sections. Its principal part is the Laplacian, so it is elliptic, and it is manifestly self adjoint.
\end{defn}

\begin{remark}\label{Lperp} Any section can be decomposed into tangent and normal components. As just observed,
$L_\psi$ maps a tangent section $Z$ to the tangent section $J_\psi Z$, so an alternative way of characterizing
$L_\psi$
is by specifying how it acts on normal sections,  $\alpha\psi$ where
$\alpha:\Sigma\ra\R$. A short calculation, using harmonicity of $\psi$, yields
\beq
L_\psi(\alpha\psi)=-(\Delta\alpha)\psi-4(\alpha_x\psi_x+\alpha_y\psi_y).
\eeq
It follows immediately that $\ker L_\psi=\ker J_\psi\oplus\ip{\psi}$. 
Note that, in general, $L_\psi$ does not map normal sections to normal sections.
\end{remark}

Now, for any $Y$ satisfying the pointwise constraint,
\beq\label{jhatdef}
A_\psi^\dagger Y=\eps^2\{|Y|^2\Delta\psi+2(Y\cdot Y_x)\psi_x+2(Y\cdot Y_y)\psi_y\}=:\wh{j}
\eeq
and so, for any curve $\psi(t)$, if $\phi$ is a wave map then $Y$ satisfies the PDE
\beq\label{YPDE}
Y_{tt}+L_\psi Y=k+\eps j'
\eeq
where $j'=j+\wh{j}$. Note that $j'$ has the same qualitative analytic properties as $j$ (specifically, 
no higher than first derivatives of $Y$ appear).

We have yet to specify the curve $\psi(t)$ in $\M_n$. We do this by demanding that the error section $Y$
should at all times be $L^2$ orthogonal to $T_{\psi}\M_n$. In this case ($\Sigma=T^2$, $n\geq 2$),
$T_\psi \M_n=\ker J_\psi$ so, in terms of the local coordinate system $q$ on $\M_n$ provided by Proposition \ref{psiprop},
this amounts to requiring
\beq\label{oc}
\ip{Y,\psi_\mu}=0,\qquad \mu=1,2,\ldots,4n.
\eeq
We convert this into an evolution equation for $q$
by differentiating the 
orthogonality constraint (\ref{oc}) twice with respect to time and using (\ref{YPDE}),
\beq
\ip{-L_\psi Y+k+\eps j',\psi_\mu}+2\eps{Y_t,\psi_{\mu\nu}}\dot{q}^\nu+\eps^2\ip{Y,\psi_{\mu\nu\lambda}}\dot{q}^\nu
\dot{q}^\lambda+\eps^2\ip{Y,\psi_{\mu\nu}}\ddot{q}^\nu=0,
\eeq
where an overdot denotes differentiation
with respect to $\tau$. Now $L_\psi$ is self adjoint and $\psi_\mu\in \ker J_\psi\subset
\ker L_\psi$, so this equation simplifies to $\ip{k,\psi_\mu}=O(\eps)$, or, more explicitly,
\beq\label{modulational}
\ddot{q}^\mu+G^\mu_{\nu\lambda}(q)\dot{q}^\nu\dot{q}^\lambda=\eps h^\mu(\eps,q,\dot{q},Y,Y_t)+\eps^2\gamma^{\mu\nu}
\ip{Y,\psi_{\nu\lambda}}\ddot{q}^\lambda
\eeq
where $\gamma$  and $G$ are the $L^2$ metric and its Christoffel symbol in the coordinate system $q$, 
(\ref{metdef}), (\ref{chrisdef}), and the function $h$ is
\bea
h^\mu&=&\gamma^{\mu\nu}\bigg\{\ip{Y_t,\psi_{\nu\lambda}}\dot{q}^\lambda
+\eps\ip{Y,\psi_{\nu\lambda\rho}}\dot{q}^\lambda\dot{q}^\rho
+\eps\ip{(\psi_\lambda\cdot\psi_\rho)Y,\psi_\nu}\dot{q}^\lambda\dot{q}^\rho\nonumber\\
&&-2\eps\ip{(\psi_x\cdot Y_x+\psi_y\cdot Y_y)Y,\psi_\nu}
+2\eps^2\ip{(\psi_\lambda\cdot Y_t)Y,\psi_\nu}\dot{q}^\lambda\nonumber\\
&&+\eps^3\ip{(|Y_t|^2-|Y_x|^2-|Y_y|^2)Y,\psi_\nu}\bigg\}.\label{hdef}
\eea
Taking the formal limit $\eps\ra 0$, (\ref{modulational}) reduces to the geodesic equation on $(\M_n,\gamma)$,
as one would hope. 

To summarize, if $\phi$ is a wave map, and $q(t)$ is a curve in $\M_n$ such that $Y=\eps^{-2}(\phi-\psi(q))$ satisfies the orthogonality constraint (\ref{oc}) at all times, then $(Y,q)$ satisfies the coupled system
\beq
\label{cs}
Y_{tt}+LY=k+\eps j',\qquad
\ddot{q}^\mu+G^\mu_{\nu\lambda}\dot{q}^\nu\dot{q}^\lambda=\eps h^\mu+\eps^2\gamma^{\mu\nu}
\ip{Y,\psi_{\nu\lambda}}\ddot{q}^\lambda
\eeq
and the pointwise constraint (\ref{pc}). Conversely, if $(Y,q)$ satisfies the constraints (\ref{pc}) and (\ref{oc})
and the coupled system (\ref{cs}), then $\phi=\psi(q)+\eps^2Y$ is a wave map. Our goal is to prove that
(\ref{cs}) with fixed initial data $q(0)=q_0$, $\dot{q}(0)=q_1$, $Y(0)=0$, $Y_t(0)=0$
has solutions with $\|Y\|_{C^1}$ bounded uniformly in $\eps$ for times of order $\eps^{-1}$. 
It follows immediately that, in the limit $\eps\ra 0$, $\phi(\tau/\eps)$ converges uniformly to a curve 
$\psi(\tau)$ in $\M_n$. In the course of the proof, we will simultaneously show that $\psi(\tau)$ is the 
geodesic with initial data $q_0,q_1$.

\section{Analytic prelimaries}\news
\label{sec:prelim}

In this section we set up the function spaces we will use, and collect some standard functional analytic
results which we will appeal to repeatedly. 
More details can be found in \cite{donkro}, and references therein.
Let $\hh^k$ denote the set of real-valued functions on $\Sigma$
whose partial derivatives up to order $k$ are square integrable. This is a Hilbert space with respect to the
inner product
\beq
\ip{f,g}_k=\sum_{|\alpha|\leq k}\int_\Sigma D_\alpha f D_\alpha g
\eeq
where $\alpha$ is a multi-index taking values from $\{x,y\}$, $|\alpha|$ is its length and $D_\alpha=\cd_{\alpha_1}
\cd_{\alpha_2}\cdots\cd_{\alpha_{|\alpha|}}$, so $D_{(x,x,y)}=\cd^2_x\cd_y$, for example. We denote the
corresponding norm by $\|\cdot\|_k$,
\beq
\|f\|_k^2=\ip{f,f}_k.
\eeq
 Let $H^k=\hh^k\oplus\hh^k\oplus\hh^k$, the space of
$\R^3$-valued functions on $\Sigma$ whose components are in $\hh^k$. This is a Hilbert space with respect
to the inner product
\beq
\ip{Y,Z}_k=\ip{Y_1,Z_1}_k+\ip{Y_2,Z_2}_k+\ip{Y_3,Z_3}_k
\eeq
whose norm will again be denoted $\|\cdot\|_k$.
We adopt the convention that $\|\cdot\|=\|\cdot\|_0$ and $\ip{\cdot,\cdot}=\ip{\cdot,\cdot}_0$, that is,
undecorated norms and inner products refer to $L^2$. We will frequently, and without further comment, 
use the Cauchy-Schwarz inequality 
\beq
\ip{Y,Z}_k\leq\|Y\|_k\|Z\|_k
\eeq
and the trivial bound $\|Y\|_k\geq \|Y\|_{k'}$ if $k\geq k'$.

In the sequel, we will prove existence of a solution
of the coupled system (\ref{cs}) with $Y\in H^k$, $Y_t\in H^{k-1}$, for $k=3$. This choice of $k$ is
motivated by the following fundamental fact about $\hh^k$ on a compact 2-manifold.

\begin{prop}[Algebra property of $\hh^k$, $k\geq 2$]\label{algebra}
The Banach space $(\hh^k,\|\cdot\|_k)$ is a Banach algebra for all $k\geq 2$. That is, if $f\in\hh^k$ and
$g\in\hh^k$ then $fg\in\hh^k$, and there exists a constant $\alpha_k>0$, depending only on $\Sigma$ and $k$,
such that $\|fg\|_k\leq\alpha_k\|f\|_k\|g\|_k$.
\end{prop}

It follows directly from this that, if $(Y,Y_t)\in H^3\oplus H^2$, then the nonlinear term $j'$ in the
coupled system is in $H^2$, and $\|j'\|_2$ can be bounded by a polynomial in $\|Y\|_3,\|Y_t\|_2$ (the point being
that $j'$ contains no derivatives of $Y$ higher than first, and $Y_x,Y_y,Y_t$ are all in $H^2$). This is crucial,
not only
for proving the local existence result for (\ref{cs}), but also in later sections where we prove that
$\|Y\|_3$ is controlled by $\ip{LY,LLY}$, and make energy estimates for the solution. 
So the fact that we
choose $k=3$ is not just motivated by a desire to get strong bounds on the error
section $Y$; the method will not work for any lower $k$. Indeed, to uniformly bound $Y$ on $\Sigma$, it would
suffice to control $\|Y\|_2$, as we have the following Sobolev inequality.

\begin{prop}[Sobolev inequalities]\label{sobolev}
Let $C^k$ denote the Banach space of continuous maps $\Sigma\ra\R^3$ with the usual norm
$\|Y\|_{C^k}=\sup\{|D_\alpha Y(p)|\: :\: |\alpha|\leq k, p\in \Sigma\}$. Then
$H^2\subset C^0$, $H^3\subset C^1$, and there
is a constant $\alpha>0$, depending only on $\Sigma$, such that $\|Y\|_{C^k}\leq\alpha\|Y\|_{k+2}$ for all $Y\in H^2$, $k=0,1$.
(More briefly, the inclusions $\iota:H^2\ra C^0$ and $\iota:H^3\ra C^1$ are continuous.)
\end{prop}

In later sections we will need to bound $\|L_\psi Y\|_k$ in terms of $\|Y\|_{k+2}$. Of course, since $L_\psi$
is a linear second order operator we have trivially, for all $q\in K$, the upper bound
\beq
\|L_\psi Y\|_k\leq C\|Y\|_{k+2}
\eeq
where $C$ is a constant depending only on $\Sigma$ and $K$. For a lower bound, we use the fact that
$L_\psi$ is elliptic.

\begin{prop}[Standard elliptic estimate]\label{elliptic}
Let $D$ be an elliptic linear differential operator of order $r$ acting on sections of a vector
bundle $V$ over $\Sigma$. Then
there exist constants $\alpha_k,\beta_k$ depending only on $\Sigma$ and $k$, such that
$$
\|DY\|_k+\alpha_k\|Y\|_0\geq \beta_k\|Y\|_{k+r}.
$$
If we consider only sections which are $L^2$ orthogonal to $\ker D$, the same inequality holds with $\alpha_k=0$.
\end{prop}

The reason for quoting this result in the context of a general
vector bundle $V$ over $\Sigma$ is that we will want to apply
it to both the ordinary Laplacian on $V=\ul{\R^3}$, and the classical Jacobi operator $J_\psi$ on
$V=\psi^{-1}TS^2\subset\ul{\R^3}$, where the $H^k$ norm is defined by inclusion.

\section{Local existence theorem}\news
\label{sec:locex}

\begin{thm}[Local existence for the coupled system]\label{locex}
Consider the coupled system (\ref{cs}) with initial data $q(0)=q_0\in K$, $q_t(0)=\eps q_1\in\R^{4n}$,
$Y(0)=Y_0\in H^3$, $Y_t(0)=Y_1\in H^2$ such that
$$
\dist(q_0,\cd K)>d,\qquad
|q_1|,\|Y_0\|_3,\|Y_1\|_2<\Gamma
$$
where $\Gamma,d$ are positive constants. Then there exist constants $C(K)>0$ and $T(K,\Gamma,d)>0$ such that
for all $\eps\in (0,C(K)/\sqrt{\Gamma})$, this initial value problem has a unique solution on $[0,T]$ with
\ben
q&\in& C^3([0,T],K)\\
Y&\in&C^0([0,T],H^3)\cap C^1([0,T],H^2)\cap C^2([0,T],H^1).
\een
If the initial data are tangent to the $L^2$ orthogonality constraint (\ref{oc}) and the pointwise constraint 
(\ref{pc}) then the solution preserves these constraints.
\end{thm}

We will prove this using Picard's method: we iteratively define a sequence $(q^i,Y^i)\in C^0([0,T],K\times H^3)$
which converges to a solution of the initial value problem. To establish that the iteration scheme is
well-defined and convergent, the following standard energy estimate for the driven wave equation is key:

\begin{thm}[Existence and energy estimate for the wave equation]\label{waveest}
The driven wave equation on $[0,T]\times\Sigma$
$$
Y_{tt}-\Delta Y=\Xi
$$
with $\Xi:[0,T]\times\Sigma\ra\R^3$  smooth and smooth initial data $Y_0=Y(0)$ and $Y_1=Y_t(0)$ has a unique global
solution. The solution is smooth, and there exists an absolute constant $c(\Sigma)\geq 1$, depending only on the
choice of torus $\Sigma$, such that
$$
\max\{\|Y_t(t)\|_2,\|Y(t)\|_3\}\leq c(\Sigma)e^{t}\left\{\|Y_1\|_2+\|Y_0\|_3+\left(\int_0^t\|\Xi(s)\|_2^2ds
\right)^\frac12\right\}.
$$
\end{thm}

\begin{proof} Existence, uniqueness and smoothness follow from \cite{joh}. Let 
$E(t)=\|Y_t(t)\|_0^2+\|Y(t)\|_1^2$. Then
\bea
E'(t)&=&2\ip{Y_t,\Delta Y+\Xi}_0+2\ip{Y,Y_t}_1
=2\ip{Y_t,\Xi}_0+2\ip{Y_t,Y}_0\nonumber\\
&\leq& 2\|Y_t\|_0^2+\|Y\|_0^2+\|\Xi\|_0^2
\leq 2E(t)+\|\Xi(t)\|_0^2\label{carkir}\\
\Rightarrow\quad\frac{d\:}{dt}(e^{-2t}E(t))&\leq&e^{-2t}\|\Xi(t)\|_0^2\leq\|\Xi(t)\|_0^2\nonumber\\
\label{E1}
\Rightarrow\quad E(t)&\leq&e^{2t}\left[E(0)+\int_0^t\|\Xi\|_0^2\right].
\eea
Now consider $Z=\Delta Y$. This is also smooth and satisfies the wave equation with source $\Delta\Xi$. Applying the above estimate to $Z$ yields
\ben
\|\Delta Y_t\|_0^2+\|\Delta Y\|_1^2&\leq& e^{2t}\left[\|\Delta Y_1\|_0^2+\|\Delta Y_0\|_1^2+
\int_0^t\|\Delta\Xi\|_0^2\right]
\leq 2e^{2t}\left[\|Y_1\|_2^2+\|Y_0\|_3^2+
\int_0^t\|\Xi\|_2^2\right].
\een
Since $\Delta$ is an elliptic operator, there exist positive constants $\alpha_k,\beta_k$ depending only on $k$ and $\Sigma$ such that
$$
\|\Delta Y\|_k^2+\alpha_k\|Y\|_0^2\geq \beta_k\|Y\|_{k+2}^2,
$$
by the standard elliptic estimate, Proposition \ref{elliptic}. The result immediately follows.
\end{proof}

To prove Theorem \ref{locex} we must first write the coupled system (\ref{cs}) as an explicit evolution system 
(note that both equations have $\ddot{q}$ on the right hand side). Let $X=\R^{4n}\times\R^{4n}\times H^3\times
H^2$ given the norm $\|(q,p,Y,Z)\|_X=\max\{|q|,\eps^{-1}|p|,\|Y\|_3,\|Z\|_2\}$. The $\eps$ dependence of the norm
is chosen so that $\|(0,q_t,0,0)\|_X=\dot{q}$.
Given any $\Gamma>0$ 
let $X_\Gamma=\{(q,p,Y,Z)\in X\: :\: q\in K,\|(0,p,Y,Z)\|\leq 8c(\Sigma)\Gamma\}$ where $c(\Sigma)\geq 1$ is
the absolute constant obtained from Theorem \ref{waveest}. Note that $X_\Gamma$ is a closed subset of a Banach space, and hence is a complete metric space with respect to the metric induced by $\|\cdot\|_X$.
Consider the matrix valued function $M:\R\times X_\Gamma\ra End(\R^{4n})$,
\beq\label{suspow}
M(\eps,q,Y)^\mu_{\: \nu}=\delta^\mu_{\: \nu}-\eps^2\gamma^{\mu\lambda}\ip{Y,\psi_{\lambda\nu}}.
\eeq
Since the matrix $(\gamma^{\mu\lambda})$ is postive definite, $K$ is compact, all $q$-dependence is smooth, and
$\|Y\|_0\leq 8c(\Sigma)\Gamma$ there exists a constant $c(K)>0$ such that $M:[0,c(K)/\sqrt{\Gamma}]\times X_\Gamma
\ra GL(4n,\R)$ and $M^{-1}:[0,c(K)/\sqrt{\Gamma}]\times X_\Gamma
\ra GL(4n,\R)$ is $C^1$ and bounded. Hence, for all $\eps\in[0,\eps_*(K,\Gamma)]$, where $\eps_*=c(K)/\sqrt{\Gamma}$
the coupled system can be rewritten
\bea
q_{tt}&=&\eps^2 f(\eps,q,q_t,Y,Y_t)\label{qf}\\
Y_{tt}-\Delta Y&=& g(\eps,q,q_t,Y,Y_t)\label{Yg}
\eea
where 
\bea
f(\eps,q,q_t,Y,Y_t)&=&M^{-1}(\eps,q,Y)(-G(q,q_t,q_t)+\eps h(\eps,q,\eps^{-1}q_t,Y,Y_t))\nonumber\\
G(q,u,v)^\mu&=&G^\mu_{\nu\lambda}(q)u^\nu v^\lambda\label{lato}\\
g(\eps,q,q_t,Y,Y_t)&=&-B_\psi Y-\psi_\mu f^\mu-\psi_{\mu\nu}
\frac{q^\mu_t}{\eps}\frac{q^\nu_t}{\eps}+\eps j'(\eps,q,\eps^{-1}q_t,Y,Y_t),\label{Gdef}
\eea
$B_\psi$ denotes the first and zeroth order piece of $L_\psi$, so $L_\psi=-\Delta+B_\psi$, explicitly
\bea
B_\psi Y&=& -(|\psi_x|^2+|\psi_y|^2)Y-2(\psi_x\cdot Y_x+\psi_y\cdot Y_y)\psi-2(\psi\cdot Y)\Delta\psi\nonumber\\
&&-2(\psi\cdot Y)_x\psi_x-2(\psi\cdot Y)_y\psi_y,\label{Bdef}
\eea
 and $h$ and
$j'=j+\wh{j}$ are as defined in (\ref{hdef}), (\ref{jdef}), (\ref{jhatdef}).

It is convenient henceforth to consider $\eps$ as a fixed parameter in $[0,\eps_*(K,\Gamma)]$
and supress the dependence of $f,g$ on $\eps$. The proof of existence will use Picard's method, which requires that $f,g$ 
be bounded and Lipschitz on $X_\Gamma$. This follows quickly from the following proposition, whose proof is straightforward but
lengthy, and so is deferred to the appendix:
\begin{proposition}\label{fgprops}
The functions $f,g$ are continuously differentiable maps $f:X_\Gamma\ra\R^{4n}$ and $g:X_\Gamma\ra H^2$.
Their differentials $\d f: X_\Gamma\ra {\cal L}(X,\R^{4n})$, $\d g: X_\Gamma\ra {\cal L}(X,H^2)$ are bounded, uniformly in $\eps$.
That is, there exist constants $\Lambda_f(K,\Gamma),\Lambda_g(K,\Gamma)>0$ such that
$$
|\d f_x \omega|\leq \Lambda_f\|\omega\|_X,\qquad
\|\d f_x \omega\|_2\leq \Lambda_g\|\omega\|_X
$$
for all $x\in X_\Gamma$, $\omega\in X$.
\end{proposition}
Note that, for Banach spaces $B,C$, ${\cal L}(B,C)$ denotes the space of bounded linear maps $B\ra C$, which is itself a Banach space
with respect to the norm $\|S\|_{{\cal L}(B,C)}=\sup\{\|S(x)\|_C/\|x\|_B\: :\: x\in B, x\neq 0\}$. 
\begin{corollary}\label{fglip}
The functions  $f:X_\Gamma\ra\R^{4n}$ and $g:X_\Gamma\ra H^2$ are Lipschitz and bounded, uniformly in $\eps$.  That is, there
exist constants $\Lambda_f,\Lambda_g,C_f,C_g$, depending only on $K$ and $\Gamma$, such that for all $x,x'\in X_\Gamma$,
$$
|f(x)-f(x')|\leq\Lambda_f\|x-x'\|_X,\quad
\|g(x)-g(x')\|_2\leq\Lambda_g\|x-x'\|_X,\quad
|f(x)|\leq C_f,\quad
\|g(x)\|_2\leq C_g
$$
\end{corollary}

\begin{proof} Let $x_1,x_2\in X_\Gamma$. Since $X_\Gamma$ is convex, the curve $x(t)=x_1+t(x_2-x_1)$, $0\leq t\leq 1$ remains in $X_\Gamma$. Hence
\bea
|f(x_1)-f(x_2)|&=&\left|\int_0^1 \d f_{x(t)}(x_2-x_1)\, dt\right|
\leq\int_0^1\Lambda_f\|x_2-x_1\|_X\, dt=\Lambda_f\|x_1-x_2\|_X.
\eea
From the definition of $f$ one sees that $f(q,0,0,0)=0$ for all $q\in K$. Hence, for all $x=(q,p,Y,Z)\in X_\Gamma$
\bea
|f(x)|&=&|f(q,p,Y,Z)-f(q,0,0,0)|\leq \Lambda_f\|(0,p,Y,Z)\|_X\leq 8c(\Sigma)\Gamma\Lambda_f.
\eea
The proof for $g$ follows mutatis mutandis.
\end{proof}

To establish uniqueness of the solution, and  to show that $q$ is three times continuously differentiable, we will
need the following extension property
of $\d f$ and $\d g$, whose proof is also deferred to the appendix:
\begin{proposition}\label{fgdiffprop}
The differentials $\d f:X_\Gamma\ra
{\cal L}(X,\R^{4n})$ and
$\d g: X_\Gamma\ra{\cal L}(X,H^2)$ of $f$ and $g$ 
extend continuously to maps $\d f^{ext}:X_\Gamma\ra{\cal L}(\R^{4n}\times \R^{4n}\times H^1\times
L^2,\R^{4n})$ and $\d g^{ext}:X_\Gamma\ra{\cal L}(\R^{4n}\times\R^{4n}\times H^1\times L^2,L^2)$,
bounded by $\Lambda_f$ and $\Lambda_g$ respectively.
\end{proposition}

The rest of this section is devoted to the proof of Theorem \ref{locex}. Let $T>0$ be chosen such that
\bea
&&T\leq \log 2,\quad
T\leq\frac1\eps,\quad
T\leq \frac{d}{\eps(2\Gamma+C_f(K,\Gamma))},\quad
T\leq \frac{\Gamma}{\eps C_f(K,\Gamma)},\nonumber\\
&&\sqrt{T}\leq \frac{\Gamma}{C_g(K,\Gamma)},\quad
T\leq\frac{1}{4\eps\Lambda_f(K,\Gamma)},\quad
\sqrt{T}\leq\frac{1}{8c(\Sigma)\Lambda_g}.
\eea
Given a complete subset $B$ of a Banach space with norm $\|\cdot\|_B$, denote by $C_TB$ the space of continuous
maps $[0,T]\ra B$ equipped with the sup norm $\n{b}=\sup\{\|b(t)\|_B\: :\: t\in[0,T]\}$. $(C_TB,\n{\cdot})$
is itself a complete subset of a Banach space. 

\subsection{Definition of the iteration scheme}

We will produce a sequence $\omega^i\in C_TX_\Gamma$ converging
to a solution of the initial value problem for the coupled system. Choose and fix
$\delta\in(0,\Gamma/4)$, and let $Y_0^i,Y_1^i\in C^\infty(\Sigma,\R^3)$ be sequences such that
\beq
\|Y_0^i-Y_0\|_3<\frac{\delta}{2^i},\qquad
\|Y_1^i-Y_1\|_2<\frac{\delta}{2^i}.
\eeq
Such sequences exist since $C^\infty$ is dense in $H^k$ for all $k\geq0$. Let 
$\omega^0=(q_0,\eps q_1,Y_0^0,Y_1^0)$, which is constant in $t$ and smooth on $\Sigma$, and trivially lies in 
$C_TX_\Gamma$. Given $\omega^i$, we define the next iterate to be the solution of the initial
value problem $q^{i+1}(0)=q_0,q_t^{i+1}(0)=\eps q_1,Y^{i+1}(0)=Y_0^{i+1},Y_t^{i+1}(0)=Y_1^{i+1}$ for
\bea
q_{tt}^{i+1}&=&\eps^2 f(\omega^i)\\
Y_{tt}^{i+1}-\Delta Y^{i+1}&=& g(\omega^i).
\eea
We must first check that the sequence $\omega^i$ is well defined. So, assume that $\omega^i$ is smooth and
lies in $C_TX_\Gamma$. Then 
\beq\label{qformula}
q^{i+1}(t)=q_0+\eps t q_1+\eps^2\int_0^t\left(\int_0^s f(\omega^i(r))dr\right)ds,
\eeq
which exists since $f\circ\omega^i$ is continuous. Now
\beq
|q^{i+1}(t)-q_0|\leq \eps T|q_1|+\frac12\eps^2T^2C_f(K,\Gamma)\leq  \eps T\left(\Gamma+\frac{C_f(K,\Gamma)}{2}\right)
\leq \frac{d}{2}
\eeq
by our choice of $T$, so $q^{i+1}(t)$ remains in $K$. Further,
\beq
\eps^{-1}|q^{i+1}_t(t)|\leq|q_1|+\eps\int_0^t|f(\omega^i(s))|ds\leq\Gamma+T\eps C_f(K,\Gamma)\leq 2\Gamma
<8c(\Sigma)\Gamma
\eeq
by our choice of $T$. Turning to $Y^{i+1}$, we see by inspection that if $Y^i,q^i$ are smooth, then
$g(\omega^i)$ is smooth, so the solution $Y^{i+1}(t)$ exists, is unique and smooth, by Theorem \ref{waveest}, which
also yields the energy estimate
\bea
\max\{\|Y_t^{i+1}(t)\|_2,\|Y^{i+1}(t)\|_3\}&\leq& C(\Sigma)e^t\{\|Y_0^{i+1}\|_3+\|Y_1^{i+1}\|_2+
\sqrt{t}C_g(K,\Gamma)\}\nonumber\\
&\leq& 
2C(\Sigma)\{2(\Gamma+\delta)+\sqrt{T}C_g(K,\Gamma)\}<8C(\Sigma)\Gamma
\eea
by our choice of $T$ and $\delta$. Hence, if $\omega^i$ is smooth and in $C_TX_\Gamma$, so is $\omega^{i+1}$.
We have already observed that $\omega^0$ is smooth and in $C_TX_\Gamma$, so, by induction, the sequence
$\omega^i\in C_TX_\Gamma$ is well-defined.

\subsection{Convergence of the iteration scheme}

We will now show that $\omega^i$ is Cauchy, and hence converges in $C_TX_\Gamma$. From (\ref{qformula}) one has
\bea
|q^{i+1}(t)-q^i(t)|&=&\eps^2\left|\int_0^t\int_0^s(f(\omega^i(r))-f(\omega^{i-1}(r)))dr\, ds\right|\nonumber\\
&\leq&\frac{\eps^2}{2}T^2\Lambda_f(K,\Gamma)\n{\omega^i-\omega^{i-1}}
\leq \frac18\n{\omega^i-\omega^{i-1}}
\eea
by our choice of $T$. Similarly
\bea
\eps^{-1}|q^{i+1}_t(t)-q^i_t(t)|&=&\eps\left|\int_0^t(f(\omega^i(r))-f(\omega^{i-1}(r)))dr\right|\nonumber\\
&\leq& \eps T\Lambda_f(K,\Gamma)\n{\omega^i-\omega^{i-1}}
\leq \frac14\n{\omega^i-\omega^{i-1}}.
\eea
Now $Z=Y^{i+1}-Y^i$ satisfies the wave equation with source $g(\omega^i)-g(\omega^{i-1})$ and small smooth initial
data $\|Z(0)\|_3,\|Z_t(0)\|_2\leq \delta/2^{i-1}$. Hence, by Theorem \ref{waveest}, for each $t\in[0,T]$,
\bea
\max\{\|Z(t)\|_3,\|Z_t(t)\|_2\}&\leq&
c(\Sigma)e^t\left\{\frac{\delta}{2^{i-2}}+\left(\int_0^t\|g(\omega^i(s))-g(\omega^{i-1}(s))\|_2^2ds\right)^\frac12
\right\}\nonumber\\
&\leq& 2c(\Sigma)\left\{\frac{\delta}{2^{i-2}}+\sqrt{T}\Lambda_g(K,\Gamma)\n{\omega^i-\omega^{i-1}}\right\}
\nonumber\\
&\leq& \frac{c(\Sigma)\delta}{2^{i-3}}+\frac14\n{\omega^i-\omega^{i-1}}.
\eea
Assembling these inequalities, one sees that
\beq
\n{\omega^{i+1}-\omega^i}\leq
\frac14\n{\omega^i-\omega^{i-1}}+\frac{\alpha}{2^i}
\eeq
where $\alpha=c(\Sigma)\delta/8$. It follows that
\beq
\n{\omega^{i+1}-\omega^i}\leq\frac{1}{4^i}\n{\omega^1-\omega^0}+\frac{\alpha}{2^{i-1}},
\eeq
and hence, for all $k\geq 1$,
\bea
\n{\omega^{i+k}-\omega^i}&\leq&\sum_{j=1}^k\n{\omega^{i+j}-\omega^{i+j-1}}
\leq \frac{1}{4^i}\n{\omega^1-\omega^0}\sum_{j=1}^\infty\frac{1}{4^j}+\frac{\alpha}{2^i}\sum_{j=0}^\infty\frac{1}{2^j}.
\eea
Hence $\omega^i$ is Cauchy with respect to $\n{\cdot}$,  so $\omega^i\ra\omega=(q,p,Y,Z)\in C_TX_\Gamma$.

\subsection{The limit solves the initial value problem}

We have established that 
\bea
q^i&\ra& q\quad\mbox{in}\quad C_TK,\\
q^i_t&\ra& p\quad\mbox{in}\quad C_T\R^{4n},\\
Y^i&\ra& Y\quad\mbox{in}\quad C_TH_3,\\
Y_t^i&\ra& Z\quad\mbox{in}\quad C_TH^2.
\eea
Now, for all $i$,
\bea
\|\omega(0)-(q_0,\eps q_1,Y_0,Y_1)\|_X&\leq&
\|\omega(0)-\omega^i(0)\|_X+\|\omega^i(0)-(q_0,\eps q_1,Y_0,Y_1)\|_X\nonumber\\
&\leq&
\n{\omega-\omega^i}+\frac{\delta}{2^i}\ra 0
\eea
as $i\ra\infty$. Hence $\omega(0)=(q_0,\eps q_1,Y_0,Y_1)$, that is, the limit has the correct initial data.

We will now show that the limit solves the coupled system and has the differentiability
properties claimed. Let $\wt{Y}(t)=Y_0+\int_0^t Z(s)ds$. Note that $\wt{Y}$ is manifestly in $C^1([0,T],H^2)$,
with derivative $\wt{Y}_t=Z$. Now
\bea
\n{Y^i-\wt{Y}}_{C_TH^2}&=&\n{Y^i(0)-Y_0+\int_0^t(Y^i_t(s)-Z(s))ds}_{C_TH^2}\nonumber\\
&\leq& \frac{\delta}{2^i}+T\n{Y^i_t-Z}_{C_TH^2}\ra 0
\eea
as $i\ra\infty$. Hence, $Y^i\ra \wt{Y}$ in $C_TH^2$. But $Y^i\ra Y$ in $C_TH^3$, hence also in $C_TH^2$, so
$Y=\wt{Y}$. Hence, $Y\in C^1([0,T],H^2)$ and $Y_t=Z$.

Consider $Y^{i+k}_{tt}-Y^i_{tt}$. This is smooth, and satisfies the wave equation with source
$g(\omega^{i+k-1})-g(\omega^{i-1})$. Hence
\bea
\|Y^{i+k}_{tt}-Y^i_{tt}\|_1&=&\|\Delta(Y^{i+k}-Y^i)+g(\omega^{i+k-1})-g(\omega^{i-1})\|_1\nonumber\\
&\leq&
2\|Y^{i+k}-Y^i\|_3+\|g(\omega^{i+k-1})-g(\omega^{i-1})\|_2\nonumber\\
&\leq&
2\|Y^{i+k}-Y^i\|_3+\Lambda_g(K,\Gamma)\|\omega^{i+k-1}-\omega^{i-1})\|_X\nonumber\\
\Rightarrow\quad
\n{Y^{i+k}_{tt}-Y^i_{tt}}_{C_TH^1}&\leq&[2+\Lambda_g(K,\Gamma)]\n{\omega^{i+k-1}-\omega^{i-1}}_{C_TX_\Gamma}.
\eea
Since $\omega^i$ is Cauchy in $C_TX_\Gamma$, it follows that $Y^i_{tt}$ is Cauchy in $C_TH^1$. Hence
$Y^i_{tt}\ra W$ in $C_TH^1$. Let $\wt{Z}(t)=Y_1+\int_0^tW$. Note that $\wt{Z}$ is manifestly in
$C^1([0,T],H^1)$ and $\wt{Z}_t=W$. Now
\bea
\n{Y_t^i-\wt{Z}}_{C_TH^1}&\leq&\n{Y_1^i-Y_1}_{C_TH^1}+\n{\int_0^t(Y^i_{tt}-W)}_{C_TH^1}\nonumber\\
&\leq&\frac{\delta}{2^i}+T\n{Y^i_{tt}-W}_{C_TH^1}\ra 0
\eea
as $i\ra\infty$. Hence $Y_t^i\ra\wt{Z}$ in $C_TH^1$. But $Y_t^i\ra Z$ in $C_TH^2$, hence also in $C_TH^1$,
so $Z=\wt{Z}$. But $Y_t=Z$. Hence, $Y\in C^2([0,T],H^1)$ and $Y_{tt}=W$.

By similar reasoning, $q_{tt}^i\ra m$ in $C_T\R^{4n}$ and $q\in C^2([0,T],\R^{4n})$ with $q_t=p$ and $q_{tt}=m$.

We can now show that $\omega$ solves the coupled system:
\bea
\n{Y_{tt}-\Delta Y-g(\omega)}_{C_TH^1}&\leq&\n{Y_{tt}-Y^i_{tt}}_{C_TH^1}+\n{\Delta(Y-Y^i)}_{C_TH^1}+
\n{g(\omega)-g(\omega^{i-1})}_{C_TH^1}\nonumber\\
&\leq&\n{Y_{tt}-Y^i_{tt}}_{C_TH^1}+2\n{Y-Y^i}_{C_TH^3}+\n{g(\omega)-g(\omega^{i-1})}_{C_TH^2}.
\nonumber \\
&&
\eea
Now $Y^i_{tt}\ra Y_{tt}$ in $C_TH^1$, $Y^i\ra Y$ in $C_TH^3$, $g:X_\Gamma\ra H^2$ is
continuous, and $\omega^i\ra\omega$ in $C_TX_\Gamma$, so $g(\omega^i)\ra g(\omega)$ in $C_TH^2$. Hence
\beq
\n{Y_{tt}-\Delta Y-g(\omega)}_{C_TH^1}=0.
\eeq
Similarly $\|q_{tt}-\eps^2f(\omega)\|_{C_T\R^{4n}}=0$, that is, $q_{tt}=\eps^2 f(\omega)$.

It remains to establish the higher differentiability of $q$. Differentiating the equation for $q^{i+1}_{tt}$
gives
\beq
q^{i+1}_{ttt}=\eps^2 df^{ext}_{\omega^i}\omega^i_t.
\eeq
Now $\omega^i\ra\omega$ in $C_TX_\Gamma$, $\omega^i_t\ra\omega_t$ in $C_T(\R^{4n}\times\R^{4n}\times H^2\times H^1)$, and $df^{ext}$ is continuous, so $q^{i+1}_{ttt}\ra \ell$, say, in $C_T\R^{4n}$. Let
$\wt{m}(t)=\eps^2f(q_0,\eps q_1,Y_0,Y_1)+\int_0^t\ell=q_{tt}(0)+\int_0^t\ell$. Note that 
$\wt{m}\in C^1([0,T],\R^{4n})$ and $\wt{m}_t=\ell$. Then
\bea
\|q^i_{tt}-\wt{m}\|_{C_T\R^{4n}}&=&\|\int_0^t(q^i_{ttt}-\ell)\|_{C_T\R^{4n}}\leq T\|q^i_{ttt}-\ell\|_{C_T\R^{4n}}
\ra 0
\eea
so $q^i_{tt}\ra\wt{m}$ in $C_T\R^{4n}$. But $q^i_{tt}\ra q_{tt}$ in $C_T\R^{4n}$, so $q_{tt}=\wt{m}$. Hence,
$q_{tt}\in C^1([0,T],\R^{4n})$, as claimed.

\subsection{Uniqueness of the solution}

Assume that $(\wt{q},\wt{Y})$ is another solution of (\ref{qf}), (\ref{Yg}) with the same initial data and regularity as
$(q,Y)$, and let $(p,Z)=(q-\wt{q},Y-\wt{Y})$.
Then $(p,Z)$ satisfies the system
\beq
Z_{tt}-\Delta Z=\Xi(t),\qquad p_{tt}=\eps^2\Upsilon(t)
\eeq
with initial data $Z(0)=Z_t(0)=0$, $p(0)=p_t(0)=0$, where
\beq
\Xi(t)=g(\omega)-g(\wt\omega),\quad
\Upsilon(t)=f(\omega)-f(\wt\omega)
\eeq
and $\omega=(q,q_t,Y,Y_t)$, $\wt\omega=(\wt{q},\wt{q}_t,\wt{Y},\wt{Y}_t)$. Define
\beq 
E(t)=\|Z(t)\|_1^2+\|Z_t(t)\|_0^2+|p|^2+\frac{1}{\eps^2}|p_t|^2,
\eeq
which, by the regularity properties of $(p,Z)$, is continuously differentiable, and has $E(0)=0$.
Reprising the argument in (\ref{carkir}), which requires only that $Z\in H^2$ and $Z_t\in H^1$, one sees that
\beq
E'(t)\leq 2E(t)+\|\Xi(t)\|_0^2+|\Upsilon(t)|^2.
\eeq
Now, arguing as in the proof of Corollary \ref{fglip}, with $\omega(s)=\omega+s(\wt{\omega}-\omega)$,
\beq
\|\Xi(t)\|_0=\|g(\omega)-g(\wt{\omega})|_0
=\left\|\int_0^1 \d g_{\omega(s)}(\wt\omega-\omega)\, ds\right\|_0
\leq \Lambda_g\max\{|p|,\eps^{-1}|p_t|,\|Z\|_1,\|Z_t\|_0\}
\eeq
by Proposition \ref{fgdiffprop}. Similarly $|\Upsilon(t)|\leq\Lambda_f\max\{|p|,\eps^{-1}|p_t|,\|Z\|_1,\|Z_t\|_0\}$. Hence
\beq
E'(t)\leq \kappa E(t)
\eeq
where $\kappa=2+\Lambda_f+\Lambda_g$, whence it follows that
\beq
\frac{d\:}{dt} e^{-\kappa t}E(t)\leq 0.
\eeq
So $e^{-\kappa t}E(t)$ is a nonincreasing, non-negative function which is zero at $t=0$. Hence $E(t)=0$ for all $t$, and we conclude
that $(p,Z)=(0,0)$ for all $t$, that is, $(q,Y)=(\wt{q},\wt{Y})$.

\subsection{Preservation of constraints}

Given the solution $(q,Y)$ produced above, define for each $\mu\in\{1,2,\ldots,4n\}$
\beq
a_\mu(t)=\ip{Y,\frac{\cd\psi}{\cd q^\mu}}.
\eeq
The $L^2$ orthogonality constraint is that $a_\mu(t)=0$ for all $\mu,t$.
By construction, the coupled system implies that $\ddot{a}_\mu=0$. If the initial data are tangent to the 
constraint then $a_\mu(0)=\dot{a}_\mu(0)$, and hence $a_\mu(t)=0$ for all $t\in[0,T]$.

Similarly, given the solution $(q,Y)$ produced above, define $\chi:[0,T]\times\Sigma\ra\R$ by
\beq
\chi=Y\cdot\psi(q)+\frac12\eps^2|Y|^2.
\eeq
The pointwise constraint is that
$\chi=0$ everywhere on $\Sigma$. Note that $\chi(t)\in \hh^3$ for all $t\in[0,T]$ so $\chi(t):\Sigma\ra\R$ is continuous.  
Assume that $\chi(0)=0$ and $\chi_t(0)=0$, that is the initial data
are tangent to the constraint. A straightforward, if lengthy, calculation using the coupled system 
and the harmonic map equation for $\psi$ shows that
$\chi$ satsifies the linear PDE
\bea
\chi_{tt}-\Delta\chi&=&2\eps^2\bigg\{[2\psi_x\cdot Y_x+2\psi_y\cdot Y_y-|\psi_\tau|^2+Y\cdot\Delta\psi
-2\eps\psi_\tau\cdot Y_t\nonumber\\
&&-\eps^4(|Y_t|^2-|Y_x|^2-|Y_y|^2)]\chi+(Y\cdot\psi_x)\chi_x+(Y\cdot\psi_y)\chi_y\bigg\}\nonumber\\
&=:&a\chi+b_1\chi_x+b_2\chi_y
\eea
where $a(t)\in H^2$,  $b_1(t),b_2(t)\in H^3$ for all $t$. Let $E(t)=\|\chi(t)\|_1^2+\|\chi_t(t)\|_0^2$.
Then
\bea
E'(t)&=&2\ip{\chi_t,\Delta\chi+a\chi+b_1\chi_x+b_2\chi_y}_0+2\ip{\chi_t,\chi}_1
=2\ip{\chi_t,\chi+a\chi+b_1\chi_x+b_2\chi_y}_0\nonumber\\
&\leq& \|\chi_t\|_0^2+\|\chi+a\chi+b_1\chi_x+b_2\chi_y\|_0^2
\leq \|\chi_t\|_0^2+\kappa\|\chi\|_1^2
\leq\kappa E(t)
\eea
where 
\beq
\kappa=\max_{0\leq t\leq T}8(1+\|a(t)\|_2^2+\|b_1(t)\|_2^2+\|b_2(t)\|_2^2).
\eeq
 Hence
$e^{-\kappa t}E(t)$ is a nonincreasing, non-negative function which is zero at $t=0$, so $E(t)=0$, whence  
$\|\chi(t)\|_1=0$ for all $t$. Since we already know that $\chi(t):\Sigma\ra\R$ is continuous,
it follows that $\chi=0$ everywhere. This completes the proof of Theorem \ref{locex}.

\section{Near coercivity of the improved Hessian}\news
\label{sec:coerce}

By repeatedly applying the local existence theorem, we can extend the solution of the coupled system whilever
$q$ remains in $K$ and $\eps^{-1}q_t$, $\|Y\|_3$ and $\|Y_t\|_2$ remain bounded. So to prove
long time existence, we must, among other things,
bound the growth of $\|Y\|_3$. The first step is to show that $\|Y\|_3$ is
controlled by the quadratic form $\ip{LY,LLY}$ or, more precisely, by the quadratic form $Q_2:H^3\ra\R$ defined
next.

\begin{defn} For a fixed harmonic map $\psi(q)$ we denote by $Q_1,Q_2$ the quadratic forms
\ben
Q_1:H^1\ra\R,\qquad Q_1(Y)&=&\int_\Sigma\left\{|Y_x|^2+|Y_y|^2-(|\psi_x|^2+|\psi_y|^2)|Y|^2-4(\psi_x\cdot Y_x
+\psi_y\cdot Y_y)\psi\cdot Y\right\}\\
Q_2:H^3\ra\R,\qquad Q_2(Y)&=&Q_1(LY).
\een
Note that both $Q_1$ and $Q_2$ are continuous, and that $Q_1(Y)=\ip{Y,LY}$ for all $Y\in H^2$.
It is also convenient to define the projection map $P:H^k\ra H^k$
$$
P(Y)=Y-(\psi\cdot Y)\psi
$$
which pointwise orthogonally projects $Y(p)$ to $T_{\psi(p)}S^2$. 
\end{defn}

\begin{lemma}\label{projlem} For all $Y\in H^2$, $Q_1(Y)\geq Q_1(P(Y))$.
\end{lemma}

\begin{proof} $Y=P(Y)+f\psi$, where $f=-\psi\cdot Y\in\hh^2$. Now
\bea
Q_1(Y)&=&\ip{Y,LY}=\ip{P(Y),LP(Y)}+2\ip{f\psi,LP(Y)}+\ip{f\psi,L(f\psi)}\nonumber\\
&=&\ip{P(Y),LP(Y)}+\ip{f\psi,L(f\psi)}
\eea
since $L$ is self-adjoint and maps tangent sections to tangent sections. But, as we saw in Remark \ref{Lperp},
\beq
L(f\psi)=-(\Delta f)\psi-4(f_x\psi_x+f_y\psi_y),
\eeq
so $\ip{f\psi,L(f\psi)}=-\ip{f,\Delta f}\geq 0$.
\end{proof}

If our error section $Y$ were a tangent section, the results of \cite{hasspe} would immediately imply that
$Q_1$ is coercive, that is, $Q_1(Y)\geq c(q)\|Y\|_1^2$, orthogonal to $\ker J$:

\begin{thm}[Haskins-Speight, \cite{hasspe}]\label{mokme}
There exists a  constant $c(q)>0$, depending continuously on $q$, such that
$$
Q_1(Y)\geq c(q)\|Y\|_1^2
$$
for all $Y\in H^1$ satisfying $\psi\cdot Y=0$, $L^2$ orthogonal to $\ker J_{\psi(q)}$.
\end{thm}

Unfortunately, $\psi\cdot Y\neq 0$
in our set-up, but is small (of order $\eps^2$). This means we can only establish the following ``near coercivity''
property for $Q_1$. This will suffice for our purposes, however.

\begin{thm}[Near coercivity of $Q_1$]\label{Q1thm} There exist constants $c(q),\wt{c}(q)>0$, 
depending continuously on $q$, such that
$$
Q_1(Y)\geq c(q)\|Y\|_1^2-\eps^2\wt{c}(q)\|Y\|_1\|Y\|_2^2
$$
for all $Y\in H^2$ satisfying the pointwise constraint (\ref{pc}), $L^2$ orthogonal to $\ker J_{\psi(q)}$.
\end{thm}

\begin{proof}
By Lemma \ref{projlem}, $Q_1(Y)\geq \ip{P(Y),LP(Y)}=\ip{P(Y),JP(Y)}$ since $L\equiv J$ on tangent sections.
Given any $Z\in\ker J$, $\ip{Z,P(Y)}=\ip{Z,Y+f\psi}=\ip{Z,Y}=0$, since $Z$ is pointwise orthogonal to $\psi$
and $Y$ is $L^2$ orthogonal to $\ker J$. Hence $P(Y)$ is $L^2$ orthogonal to $\ker J$, and so, by Theorem
\ref{mokme}, there exists a constant $\wt{c}(q)>0$ such that
\beq\label{alan3}
Q_1(Y)\geq \wt{c}(q)\|P(Y)\|_1^2.
\eeq
Now, since $Y$ satisfies (\ref{pc}),
\bea
\|P(Y)\|_1^2&=&\|Y+\frac12\eps^2|Y|^2\psi\|_1^2
\geq\|Y\|_1^2-\eps^2\ip{Y,|Y|^2\psi}_1
\label{alan4}
\eea
and, by the algebra property of $\hh^2$ (Proposition \ref{algebra}),
\bea
\ip{Y,|Y|^2\psi}_1&\leq&\|Y\|_1\| |Y|^2\psi\|_2
\leq C\|Y\|_1\|Y\|_2^2\|\psi\|_2\label{alan5}
\eea
where $C>0$ is a constant depending only on $\Sigma$. Combining (\ref{alan3}), (\ref{alan4}) and (\ref{alan5}), and
noting that $\|\psi\|_2$ depends continuously (in fact smoothly) on $q$, the result
immediately follows.
\end{proof}

\begin{thm}[Near coercivity of $Q_2$]\label{Q2thm} There exist constants $c(q),\wt{c}(q)>0$, depending continuously on $q$, such that
$$
Q_2(Y)\geq c(q)\|Y\|_3^2-\eps^2\wt{c}(q)(\|Y\|_3^3+\eps^2\|Y\|_3^4)
$$
for all $Y\in H^3$ satisfying the pointwise constraint (\ref{pc}), $L^2$ orthogonal to $\ker J_{\psi(q)}$.
\end{thm}

\begin{proof} Recall that the pointwise constraint (\ref{pc}) is equivalent to $|\psi+\eps^2 Y|\equiv 1$.
The set of smooth maps $\Sigma\ra S^2$ is dense in the Banach manifold of $H^3$ maps $\Sigma\ra S^2$,
$\psi$ and $\ker J$ are smooth, and $Q_2:H^3\ra\R$ is continuous,
so it suffices to prove the inequality in the case that $Y$ is smooth. So, let
$Y$ be smooth, $L^2$ orthogonal to $\ker J$ and satisfy the pointwise constraint (\ref{pc}). Define the smooth
section $Z=LY$ and the smooth
real functions $\alpha=|Y|^2$ and $\beta=Z\cdot Y$. Since $Y$ satisfies (\ref{pc}), and $\psi$ is harmonic,
it follows that
\beq
\beta=\eps^2\left\{Y\cdot\Delta Y+|Y_x|^2+|Y_y|^2+2(|\psi_x|^2+|\psi_y|^2)|Y|^2\right\}.
\eeq
In the following, $c_1(q),c_2(q),\ldots$ denote positive functions depending continuously on $q$.
We have the following elementary estimate,
\beq
\|\beta\psi\|_1^2\leq c_1(q)\|\beta\|_1^2
\leq\eps^4 c_2(q)\left\{\|Y\|_{C^1}^2\|Y\|_2^2+\|Y\|_{C^0}^2\|Y\|_3^2\right\}.
\eeq
Applying the Sobolev inequalities (Proposition \ref{sobolev}) gives
\beq
\label{betaest}
\|\beta\psi\|_1\leq\eps^2 c_3(q)\|Y\|_3^2.
\eeq
We will also need to estimate $\|\alpha\|_3$. Again, we have an elementary estimate
\beq
\|\alpha\|_3^2\leq c\left\{\|Y\|_{C^1}^2\|Y\|_2^2+\|Y\|_{C^0}^2\|Y\|_3^2\right\}
\eeq
which, on appealing to Proposition \ref{sobolev} yields
\beq\label{alphaest}
\|\alpha\|_3\leq c\|Y\|_3^2.
\eeq

By Lemma \ref{projlem}, 
\beq
Q_2(Y)=Q_1(Z)\leq Q_1(P(Z))=\ip{P(Z),JP(Z)}.
\eeq
It is in the last step that we have used the smoothness of $Y$ (assuming only $Y\in H^3$ gives
$P(Z)\in H^1$, which is not sufficiently regular to make sense of $JP(Z)$).
 Since $L$ is self-adjoint and $\ker J\subset\ker L$,
$Z=LY$ is automatically $L^2$ orthogonal to $\ker J$, as is $P(Z)=Z+\beta\psi$
(since $\psi$ is {\em pointwise} orthogonal to anything in $\ker J$). Hence, by Theorem \ref{mokme} and
the estimate (\ref{betaest})
\bea
Q_2(Y)&\geq& c_4(q)\|P(Z)\|_1^2
=c_4(q)\|Z-\beta\psi\|_1^2
\geq c_4(q)\left\{\|Z\|_1^2-\|Z\|_1\|\beta\psi\|_1\right\}\nonumber\\
&\geq& c_4(q)\left\{\|Z\|_1^2-\eps^2 c_3(q)\|Y\|_3^2\|Z\|_1\right\}\label{alan1}.
\eea
We next estimate $\|Z\|_1=\|LY\|_1$ in terms of $\|Y\|_3$.
Note that $Y$ is {\em not} $L^2$ orthogonal to $\ker L$, since it has a component in the direction
of $\psi$, so we cannot apply the standard elliptic estimate for $L$ directly. We must decompose 
\beq
Y=P(Y)+(\psi\cdot Y)\psi=P(Y)-\frac12\eps^2\alpha\psi,
\eeq
using (\ref{pc}),
 and handle the two terms separately. Then, by
Proposition \ref{elliptic} (for the lower bound on $\|JP(Y)\|_1$), and an elementary estimate (for the upper bound
on $\|JP(Y)\|_1$),
\bea
\|JP(Y)\|_1-\frac12\eps^2\|L(\alpha\psi)\|_1\leq &\|Z\|_1&\leq
\|JP(Y)\|_1+\frac12\eps^2\|L(\alpha\psi)\|_1\nonumber\\
c_5(q)\|P(Y)\|_3-\eps^2c_6(q)\|\alpha\|_3\leq &\|Z\|_1&\leq
c_7(q)\|P(Y)\|_3+\eps^2c_6(q)\|\alpha\|_3\nonumber\\
c_5(q)\|Y\|_3-\eps^2c_8(q)\|\alpha\|_3\leq &\|Z\|_1&\leq
c_7(q)\|Y\|_3+\eps^2c_8(q)\|\alpha\|_3\nonumber\\
c_5(q)\|Y\|_3-\eps^2c_9(q)\|Y\|_3^2\leq &\|Z\|_1&\leq c_7(q)\|Y\|_3+\eps^2c_9(q)\|Y\|_3^2
\label{alan2}
\eea
where we have used (\ref{alphaest}) in the last line. Combining (\ref{alan1}) and (\ref{alan2}), the result
immediately follows.
\end{proof}

\begin{remark} Since $K$ is compact, we can replace $c(q)$, $\wt{c}(q)$ in
Theorems \ref{Q1thm}, \ref{Q2thm} by global constants $C,\wt{C}>0$,
under the extra assumption that $q\in K$.
\end{remark}

\section{Energy estimates for the coupled system}\news
\label{energyestimates}

Having shown that $Q_2(Y)$ controls $\|Y\|_3^2$, for small $\eps$, we must now bound the growth of $Q_2(Y)$
for a solution $(q,Y)$ of the coupled system. We do this by establishing quasi-conservation of energies
$E_1,E_2$, related to $Q_1,Q_2$:

\begin{defn}\label{Edef} Let $(q,Y):[0,T]\ra K\times H^3$ with the regularity of Theorem \ref{locex}. Associated to
$(q,Y)$ we define the energies $E_1,E_2:[0,T]\ra\R$,
$$
E_1(t)=\frac12\|Y_t\|_0^2+\frac12Q_1(Y),\qquad
E_2(t)=\frac12\|(LY)_t\|_0^2+\frac12Q_2(Y).
$$
Note that $E_1$ is $C^1$ and $E_2$ is continuous.
\end{defn}

Throughout this section we will use the following 

\begin{conv}\label{convention}
$C$ will denote a positive constant depending (at most) on the choice of $\Sigma$ and $K$.
$c(a_1,a_2,\ldots,a_p)$ will denote a smooth positive bounding function of $p$ non-negative
real arguments, which may also depend (implicitly) on $\Sigma, K$,
and which is increasing in each of its arguments.
$C_0$, $c_0$ will denote that the constant or bounding function depends, 
in addition, on the initial data $q_0,q_1,Y_0,Y_1$.
 The value of $C,C_0,c,c_0$ may vary from line to
line.
\end{conv}

\begin{thm}[Quasi-conservation of $E_1$]\label{E1cons}
Let $(q,Y):[0,T]\ra K\times H^3$ be a solution of the coupled system with the initial data and regularity of Theorem
\ref{locex}. Then
$$
E_1(t)\leq C_0+C(|\ddot{q}|+|\dot{q}|^2)\|Y(t)\|_0+
\eps\int_0^t c(|\dot{q}|,|\ddot{q}|,|\dddot{q}|,\|Y\|_3,\|Y_t\|_2).
$$
\end{thm}

\begin{proof} The solution satisfies (\ref{cs}) and has $Y\in H^3$, $Y_t\in H^2$, $Y_{tt}\in H^1$, so
\bea
\frac{dE_1}{dt}&=&\ip{Y_t,Y_{tt}}+\ip{Y_t,LY}+\frac12\eps\ip{Y,L_\tau Y}
=\ip{Y_t,k+\eps j'}+\frac12\eps\ip{Y,L_\tau Y}\nonumber\\
&=&\frac{d\: }{dt}\ip{Y,k}-\eps\ip{Y,k_\tau-\frac12L_\tau Y}+\eps\ip{Y_t,j'}\nonumber\\
\Rightarrow
E_1(t)&=&E_1(0)-\ip{Y(0),k(0)}+\ip{Y(t),k(t)}+\eps\int_0^t\left\{\ip{Y,\frac12 L_\tau Y-k_\tau}+\ip{Y_t,j'}\right\}
\nonumber\\
&\leq& C_0+\|k(t)\|_0\|Y(t)\|_0+2\eps\int_0^t\left\{\|Y\|_0^2+\|L_\tau Y\|_0^2+\|k_\tau\|_0^2+\|Y_t\|_0^2
+\|j'\|_0^2\right\}.\nonumber \\
&&
\eea
Now, only the first and zeroth order parts of $L$ depend on time, so it is clear that
\beq
\|L_\tau Y\|_0\leq C|\dot{q}|\|Y\|_1.
\eeq
Recall that $k=-\psi_{\tau\tau}$, so
\beq
\|k(t)\|_0\leq C\|k(t)\|_{C^0}\leq C(|\ddot{q}|+|\dot{q}|^2),
\eeq
and
\beq
\|k_\tau(t)\|_0\leq C\|k_{\tau}(t)\|_{C^0}\leq C(|\dddot{q}|+|\ddot{q}|^2+|\dot{q}|^3+|\dot{q}|^2).
\eeq
Finally, it follows immediately from the algebra property of $\hh^2$ (Proposition \ref{algebra}) that 
$\|j'\|_0\leq\|j'\|_2\leq c(|\dot{q}|,\|Y\|_3,\|Y_t\|_2)$, and the result
directly follows.
\end{proof}

We will need a similar result bounding the growth of $E_2(t)$. Formally, this is obtained by applying
the argument above with $Y$ replaced by $LY$ (which formally solves a PDE of the form $(LY)_{tt}+L(LY)=Lk+O(\eps)$). Unfortunately, this argument is not rigorous since $Y$ is insufficiently regular to make sense of 
expressions like $LY_{tt}$ (recall $Y_{tt}$ is only $H^1$).  

\begin{thm}[Quasi-conservation of $E_2$]\label{E2cons}
Let $(q,Y):[0,T]\ra K\times H^3$ be a solution of the coupled system with the initial data and regularity of Theorem
\ref{locex}. Then
$$
E_2(t)\leq C_0+C(|\ddot{q}|+|\dot{q}|^2)\|Y(t)\|_2+
\eps\int_0^t c(|\dot{q}|,|\ddot{q}|,|\dddot{q}|,\|Y\|_3,\|Y_t\|_2).
$$
\end{thm}

\begin{proof} 
By uniqueness, $(q,Y)$ must arise as the limit of an iteratively defined sequence of smooth functions $(q^i,Y^i)$,
as constructed in the proof of Theorem \ref{locex}. Recall that the sections $Y^i$ satisfy the PDEs
\beq
Y^{i+1}_{tt}-\Delta Y^{i+1}+B^iY^i=k^i+\eps(j')^i
\eeq
where $B$ denotes the first and zeroth order piece of $L$ (so $L=-\Delta+B$) and the superscript $i$ on $B^i$,
$k^i$, $(j')^i$ denotes that the quantity is evaluated on the iterate $(q^i,Y^i)$. Recall also
$(Y^i,Y^i_t,Y^i_{tt})\ra
(Y,Y_t,Y_{tt})$ in $C^0(H^3\oplus H^2\oplus H^1)$. Now, for each $i$ define
\bea
Z^i&=&-\Delta Y^{i+1}+B^iY^i\nonumber\\
E^i(t)&=&\frac12\|Z^i_t\|^2+\frac12\ip{Z^i,-\Delta Z^i+B^iZ^i}.
\eea
Each $E^i:[0,T]\ra\R$ is smooth, $E^i\ra E_2$ uniformly on $[0,T]$, $Z^i\ra Z=LY$ in $H^1$ and
$Z^i_t\ra (LY)_t$ in $L^2$. Note that $Z^i$ satisfies the PDE
\beq \label{ZPDE}
Z^i_{tt}-\Delta Z^i+B^iZ^i=B^i(Z^i-Z^{i-1})+\hat{k}^i+\eps\hat{j}^i
\eeq
where $\hat{k}^i=-\Delta k^i+B^ik^{i-1}$ and $\hat{j}^i=-\Delta(j')^i+B^i(j')^{i-1}$. Since $B^i$ is self-adjoint,
one has
\bea
\frac{dE^i}{dt}&=&\ip{Z^i_t,Z^i_{tt}-\Delta Z^i+B^iZ^i}+\frac\eps2\ip{Z^i,B^i_\tau Z^i}\nonumber\\
&=&\ip{Z^i_t,B^i(Z^i-Z^{i-1})+\hat{k}^i}+\eps\left\{\ip{Z^i_t,\hat{j}^i}+\frac12\ip{Z^i,B^i_\tau Z^i}\right\}
\nonumber\\
&=&\frac{d\: }{dt}\ip{Z^i,B^i(Z^i-Z^{i-1})+\hat{k}^i}-\ip{B^iZ^i,Z^i_t-Z^{i-1}_t}\nonumber\\
&&+\eps\{\ip{Z^i,B^i_\tau(\frac12Z^i-Z^{i-1})+\hat{k}^i_\tau}+\ip{Z^i_t,\hat{j}^i}\}.
\eea
Integrating this from $0$ to $t$ and taking the limit $i\ra\infty$ yields
\bea
E_2(t)-E_2(0)&\leq&C_0+\|Z(t)\|_0\|\hat{k}(t)\|_0+C\eps\int_0^t\{\|Z\|_0(\|\hat{k}_\tau\|_0+|\dot{q}|\|Z\|_1)
+\|Z_t\|_0\|\hat{j}\|_0\}\nonumber\\ && \label{alan10}
\eea
where we have used the facts that $Z^i_t-Z^{i-1}_t\ra 0$ in $L^2$, $Z^i-Z^{i-1}\ra 0$ in $H^1$,
$\hat{k}^i\ra\hat{k}=Lk$ in $H^3$, $\hat{k}^i_\tau\ra\hat{k}_\tau$ in $H^3$ and $\hat{j}^i\ra\hat{j}=Lj'$
in $L^2$ (since $L:H^2\ra L^2$ and $j:K\times\R^{4n}\times H^3\times H^2\ra H^2$ are continuous).
Now, we have the elementary estimates
\bea
\|Z\|_k&=&\|LY\|_k\leq C\|Y\|_{k+2},\qquad k=1,2,\nonumber\\
\|Z_t\|_0&=&\|LY_t+\eps B_\tau Y\|_0\leq C(\|Y_t\|_2+\eps|\dot{q}|\|Y\|_1),\nonumber\\
\|\hat{k}\|_0&\leq&C(|\ddot{q}|+|\dot{q}|^2),\nonumber\\
\|\hat{k}_\tau\|_0&\leq&C(|\dddot{q}|+|\ddot{q}|^2+|\dot{q}|^2+|\dot{q}|^3)\nonumber\\
\|\hat{j}\|_0&\leq&C\|j'\|_2\leq c(|\dot{q}|,\|Y\|_3,\|Y_t\|_2)\nonumber.
\eea
Combining these with (\ref{alan10}), the result follows.
\end{proof}

\section{Long time existence and proof of the main theorem}\news
\label{sec:long}

Throughout this section we choose and fix $q_0\in K$ and $q_1\in\R^{4n}$, and denote by $(q,Y)$ 
the solution of the coupled system (\ref{cs})
with initial data $q(0)=q_0$, $\dot{q}(0)=q_1$, $Y(0)=Y_t(0)=0$. By Theorem \ref{locex}, 
provided $\eps<C(K)/\sqrt{|q_1|}$, this solution
exists at least for time $t\in[0,T_0]$, where $T_0$ depends on the initial data, but is independent of $\eps$.
Moreover, the solution is unique, has the advertised regularity, 
satisfies the pointwise and $L^2$ orthogonality constraints (\ref{pc}), 
(\ref{oc}), and obeys the energy estimates of section \ref{energyestimates}. Denote by $q_*(\tau)$ the geodesic
in $(\M_n,\gamma)$ with the same initial data, $q_*(0)=q_0$, $\dot{q}_*(0)=q_1$. Note that
$q_*$, considered as a function of rescaled time $\tau$, is independent of $\eps$.
Since geodesic flow
conserves speed $\gamma(\dot{q}_*,\dot{q}_*)$, which uniformly bounds $|\dot{q}_*|^2$ on $K$, 
there exist $\tau_0 >0$, $\alpha_0>0$,
depending only on the initial data,
such that $q_*$ exists and has 
\beq
|\dot{q}_*|\leq\alpha_0,\quad
|\ddot{q}_*|\leq\alpha_0,\quad
\dist(q(\tau),\cd K))<d/2,
\eeq
 for all $\tau\in[0,\tau_0]$, where $d=\dist(q_0,\cd K)$. Hence, the geodesic $q_*$ exists for time
$t\in[0,\eps^{-1}\tau_0]$ which, for $\eps$ small, exceeds $T_0$. Whilever $q$, $q_*$ both exist, we
define $\eps^2\wt{q}(t)$ to be the error between them, that is
\beq
q=q_*+\eps^2\wt{q},
\eeq
and
\beq
M(s)=\max_{0\leq t\leq s}\left\{\eps^2|\wt{q}(t)|^2+|\wt{q}'(t)|^2+|\wt{q}''(t)|^2+\|Y(t)\|_3^2+\|Y_t(t)\|_2^2\right\},
\eeq
where primes denote differentiation with respect to $t$.
This function, which measures the total error in replacing the wave map $\phi=\psi(q)+\eps^2 Y$ with the geodesic $\psi(q_*)$, is continuous, manifestly increasing, and has initial value $M(0)=0$. Our next task is to
bound its growth. Before doing so, we define another absolute constant (depending only on $K$ and $\Sigma$), 
which will appear frequently in this section:
\beq
\alpha_a=\sup\{\|\gamma^{\mu\nu}\psi_{\nu\lambda}\|\: :\: q\in K, 0\leq \mu,\lambda\leq 4n\}.
\eeq
We will again use Convention \ref{convention} regarding bounding constants and functions.

\begin{thm}[A priori bound] \label{apriori}
Whilever $(q,Y)$ exists, and $t<\tau_0/\eps$, and $M(t)<\eps^{-4}\alpha_a^{-2}$,
$$
M(t)\leq C_0+C_0M(t)^{\frac12}+(\eps^2+\eps t+\eps^2t^2+\eps^4t^4)\frac{c_0(M(t))}{1-\eps^2\alpha_a M(t)^\frac12}.
$$
\end{thm}

\begin{proof} 
We first derive the ODE satisfied by $\wt{q}$.
The curve $q(\tau)$ satisfies the ODE
\beq\label{qode}
\ddot{q}+G(q,\dot{q},\dot{q})=\eps h(\eps,q,\dot{q},Y,Y_t)+\eps^2 a(q,\ddot{q},Y)
\eeq
where $G$, $h$ are defined in (\ref{hdef}), (\ref{Gdef}), and
\beq
a:K\times\R^{4n}\times L^2\ra\R^{4n},\quad a(q,v,Y)^\mu=\gamma^{\mu\nu}\ip{\psi_{\nu\lambda},Y}v^\lambda,
\eeq
which is smooth with respect to $q$ and linear with respect to $v$ and $Y$. The geodesic satisfies the ODE
\beq\label{geod}
\ddot{q}_*+G(q_*,\dot{q}_*,\dot{q}_*)=0
\eeq
with the same initial data. Substituting $q=q_*+\eps^2\wt{q}$ into (\ref{qode}), and using (\ref{geod}), we see
that $\wt{q}$ satisfies 
\bea
\wt{q}''&=&[G(q_*,\dot{q}_*,\dot{q}_*)-G(q_*+\eps^2\wt{q},\dot{q}_*,\dot{q}_*)]+
[G(q,\dot{q}_*,\dot{q}_*-G(q,\dot{q}_*+\eps\wt{q}',\dot{q}_*+\eps q')]\nonumber\\
&&+\eps h(\eps,q,\dot{q},Y,Y_t)+\eps^2 a(q,Y,\ddot{q}_*+\wt{q}'').
\eea
Now, $q$, by assumption, remains in $K$, and for all $t\in [0,\tau_0/\eps]$, $q_*$ remains in $K$ and 
$|\dot{q}_*|\leq \alpha_0$,
so, since $G(q,u,v)$ is Lipshitz with respect to $q$ (on $K$) and bilinear in $(u,v)$,
we have
\bea
|\wt{q}''|&\leq& C\alpha_0^2\eps^2|\wt{q}|+C[\eps^2|\wt{q}'|^2+\eps\alpha_0|\wt{q}'|+\eps|\wt{q}'|]
\eps|h|+\eps^2|a|\nonumber\\
&\leq& C_0\eps M^\frac12+C\eps^2 M+\eps|h|+\eps^2|a|.
\eea
To estimate the $h$ term, we note that it is smooth in $q$ and polynomial in $\dot{q}$ and $Y$ and its (first)
derivatives, so in light of Proposition \ref{algebra} we have (for $\eps\leq 1$) the crude bound
\beq
|h(\eps,q,\dot{q},Y,Y_t)|\leq c(|\dot{q}|,\|Y\|_3,\|Y_t\|_2).
\eeq
Now, by the definition of $M$,
\beq\label{qt}
|\dot{q}|\leq |\dot{q}_*|+\eps|\wt{q}'|\leq \alpha_0+\eps M^\frac12,
\eeq
so $|h|\leq c_0(M)$. Turning to $|a|$, we have by linearity,
\beq
|a(q,Y,\ddot{q}_*+\wt{q}'')|\leq \alpha_a\|Y\|(|\ddot{q}_*|+|\wt{q}''|)\leq \alpha_a M^\frac12(\alpha_0+M^\frac12)
\leq c_0(M).
\eeq
Hence,
\beq\label{wqtt}
|\wt{q}''(t)|\leq \eps c_0(M(t)).
\eeq
Now $\wt{q}'(t)=\wt{q}'(0)+\int_0^t\wt{q}''=\int_0^t\wt{q}''$, so
\beq\label{wqt}
|\wt{q}'(t)|\leq\eps\int_0^tc_0(M(s))\, ds\leq \eps t c_0(M(t))
\eeq
since $c_0$ and $M$ are, by definition, increasing. Similarly
\beq\label{wq}
\eps|\wt{q}(t)|\leq\eps\int_0^t|\wt{q}'|\leq \eps^2 t^2 c_0(M(t)).
\eeq

We have now bounded the growth of all the $\wt{q}$ terms in $M$. To bound the growth of $\|Y\|_3$ and $\|Y_t\|_2$,
we will use the energy estimates of section
\ref{energyestimates} and the near coercivity property of $Q_2$ (Theorem
\ref{Q2thm}). But to do this, we need to control 
$|\ddot{q}|$ and $|\dddot{q}|$, which appear in the energy
estimates for $E_1(t)$ and $E_2(t)$, so we have not yet finished with the ODE for $q$. For  
$\ddot{q}$ we have from (\ref{wqtt}) the obvious bounds
\beq\label{qtt}
|\ddot{q}|\leq|\ddot{q}_*|+|\wt{q}''|\leq \alpha_0+\eps c_0(M(t)).
\eeq
For $\dddot{q}$ we must work harder.
So, for fixed $\eps$,
let $h_i$, $i=1,2,3,4$, denote the partial derivatives of $h:K\times\R^{4n}\times H^3\times H^2\ra \R^{4n}$
with respect to each of its four entries. Similarly, let $G_1$, $a_1$ denote the derivatives of $G$, $a$ with respect to their first entries. Differentiating (\ref{qode}) with respect to $\tau$ one finds
\bea
\dddot{q}+G_1(q,\dot{q},\dot{q})\dot{q}+2G(q,\dot{q},\dot{q})\ddot{q}&=&\eps{h_1\dot{q}+h_2\ddot{q}}
+h_3Y_t+h_4Y_{tt}\nonumber\\
&&+\eps^2(a_1(q,\ddot{q},Y)\dot{q}+a(q,\dddot{q},Y))+\eps a(q,\ddot{q},Y_t)
\eea
where we have used the bilinearity properties of $G$ and $a$. Inspecting the formula for $h$ (\ref{hdef}),
we obtain, using (\ref{qtt}) and (\ref{qt}) a crude bound
\bea
|\dddot{q}|&\leq& c_0(M)+c_0(M)\|Y_{tt}\|+\eps^2\alpha_a|\dddot{q}|\|Y\|\nonumber\\
&\leq& c_0(M)+c_0(M)(\|LY\|+\|k\|+\eps\|j'\|)+\eps^2\alpha_a M^\frac12|\dddot{q}|\nonumber\\
&\leq& c_0(M)+c_0(M)(\|Y\|_2+|\ddot{q}|+|\dot{q}|^2+c(|\dot{q}|,\|Y\|_3,\|Y_t\|_2))+\eps^2\alpha_a M^\frac12|\dddot{q}|\nonumber\\
&\leq& c_0(M)+\eps^2\alpha_a M^\frac12|\dddot{q}|.
\eea
Hence
\beq\label{qttt}
|\dddot{q}|\leq \frac{c_0(M)}{1-\eps^2\alpha_a M^\frac12}
\eeq
whilever $M(t)<\eps^{-4}\alpha_a^{-2}$. We can now turn to bounding $Y$.

Since $Q_2$ is nearly coercive (Theorem \ref{Q2thm}),
\beq\label{simc1}
\|Y\|_3^2\leq C\{Q_2(Y)+\eps^2 c(\|Y\|_3)\},
\eeq
and, by the standard elliptic estimate for $L$ (Proposition \ref{elliptic})
\bea
\|Y_t\|_2^2&\leq& C\{\|LY_t\|^2+\|Y_t\|^2\}
\leq C\{\|(LY)_t\|^2+\eps^2\|L_\tau Y\|^2+\|Y_t\|^2\}\nonumber\\
&\leq& C\{\|(LY)_t\|^2+\eps^2|\dot{q}|^2\|Y\|_1^2+\|Y_t\|^2\}\label{simc2}
\eea
since the principal part of $L$ does not depend on time. Adding (\ref{simc1}) and (\ref{simc2}), one sees that
\beq
\|Y\|_3^2+\|Y_t\|_2^2\leq C\{E_1(t)+E_2(t)+\eps^2c(\|Y\|_3)+\eps^2|\dot{q}|^2\|Y\|_1^2\}
\leq C\{E_1(t)+E_2(t)+\eps^2 c_0(M)\},
\eeq
where $E_1,E_2$ are as in Definition \ref{Edef}. 
Then, by Theorem \ref{E1cons} and \ref{E2cons}, and the estimates (\ref{qt}), (\ref{qtt}), (\ref{qttt}),
\bea
\|Y\|_3^2+\|Y_t\|_2^2&\leq& C_0+(C_0+\eps c_0(M))M^\frac12+\eps\int_0^t\left(1+\frac{c_0(M)}{1-\eps^2\alpha_a M^\frac12}\right) c_0(M)+\eps^2 c_0(M)\nonumber\\
&\leq& C_0+C_0M^\frac12+\eps c_0(M)+\frac{\eps tc_0(M)}{1-\eps^2\alpha_a M^\frac12},\label{simc3}
\eea
provided $\eps\leq 1$ and $M(t)<\eps^{-4}\alpha_a^{-2}$.
Combining (\ref{simc3}), (\ref{wq}), (\ref{wqt}) and (\ref{wqtt}) gives the a priori bound claimed.
\end{proof}

\begin{thm}[Long time existence]\label{longtime}
There exist $\eps_*>0$ and $\tau_*>0$, depending only on the initial data, such that for all $\eps\in (0,\eps_*)$
the solution $(q,Y)$ persists for all $t\in[0,\tau_*/\eps]$ and has $M(t)$ bounded, independent of $\eps$.
\end{thm}

\begin{proof} Choose and fix a constant $M_*>4\alpha_0^2$ so large that $M_*>C_0+C_0M_*^\frac12$, where $C_0$
is the specific constant, depending only on initial data, appearing in the a priori bound. Assume $\eps>0$
is so small that
\beq\label{epscond}
\eps^2<\frac{1}{\alpha_a M_*^\frac12},\quad
\eps<\frac{d}{4M_*^\frac12},\quad
\eps\leq\frac12,\quad
\eps<\frac{C(K)}{M_*^\frac14}
\eeq
where, as before, $d=\dist(q_0,\cd K)$ and $C(K)>0$ is the constant quoted in Theorem \ref{locex}.
Then, whilever $M(t)\leq M_*$, 
$\|Y\|_3<M_*^\frac12$, $\|Y_t\|_2\leq M_*^\frac12$,
\beq
|\dot{q}|\leq \alpha_0+\eps M_*^\frac12 <M_*^\frac12,
\eeq
and
\beq
\eps^2\|\wt{q}\|\leq \eps M^\frac12<\eps M_*^\frac12<\frac{d}{4}.
\eeq
Hence, whilever $M(t)\leq M_*$, the value of the solution $(q,q_t,Y,Y_t)(t)$ satisfies the conditions of the initial
data for the local existence theorem \ref{locex}, with $\Gamma=M_*^\frac12$ and $\dist(q(t),\cd K)<\frac{d}{4}$.
Given the last condition on $\eps$, (\ref{epscond}), it follows that we may apply Theorem \ref{locex} and
extend the solution for a time $\delta T>0$ depending only on $M_*$ and $d$, independent of $\eps$. It follows
that the solution persists for as long as $M(t)\leq M_*$. Furthermore, by the first condition on $\eps$, 
(\ref{epscond}), whilever $M(t)\leq M_*$ the solution obeys the a priori bound, Theorem \ref{apriori}.

For each $\eps>0$ let $t_\eps=\sup\{t\: :\: M(t)\leq M_*\}$. We claim that there exists $\eps_*>0$ such that
$\eps t_\eps$ is bounded away from zero on $[0,\eps_*]$. Note that this immediately implies the statement in the
theorem since then there exists $\tau_*>0$ such that $t_\eps\geq \tau_*/\eps$ for all $\eps\in(0,\eps_*)$, and the
solutions exists, with $M(t)\leq M_*$ on $[0,\tau_*/\eps]$. Assume, towards a contradiction, that no such $\eps_*$
exists. Then there is a positive sequence $\eps_i\ra 0$ such that $\eps_i t_{\eps_i}\ra 0$. 
But then the a priori bound
at time $t_{\eps_i}$ gives (recall $M$ is continuous, so $M(t_{\eps_i})=M_*$), in the limit $i\ra\infty$,
\beq
M_*<C_0+C_0 M_*^\frac12,
\eeq
a contradiction, by our choice of $M_*$.
\end{proof}

\begin{proof}[Proof of {Main Theorem}]
By Theorem \ref{longtime}, for all $\eps\in (0,\eps_*]$ the solution exists for $t\in[0,\tau_*/\eps]$, and coincides with
$\psi(q_*(\tau)+\eps^2\wt{q}(t))+\eps^2 Y(t)$, with $\eps|\wt{q}(t)|$, $\|Y(t)\|_3$, and hence (by Proposition
\ref{sobolev})
$\|Y(t)\|_{C^0}$  uniformly bounded in
$t$ and $\eps$. 
The rescaled solution is
\beq
\phi^\eps:[0,\tau_*]\times\Sigma\ra S^2\subset \R^3,\qquad
\phi^\eps(\tau,x,y)=\psi(q_*(\tau)+\eps^2\wt{q}(\tau/\eps),x,y)+\eps^2 Y(\tau/\eps,x,y),
\eeq
and the geodesic with the same initial data is
\beq
\psi_*:[0,\tau_*]\times\Sigma,\qquad \psi_*(\tau,x,y)=\psi(q_*(\tau),x,y).
\eeq
Now $\psi:K\times\Sigma\ra S^2$ is smooth, hence uniformly continuous (since $K\times\Sigma$ is
compact). Hence, as $\eps\ra 0$, $\phi^\eps$ converges uniformly on $[0,\tau_*]\times\Sigma$ to $\psi_*$.
Furthermore,
\bea
\phi^\eps_x(\tau,x,y)&=&\psi_x(q_*(\tau)+\eps^2\wt{q}(\tau/\eps),x,y)+\eps^2 Y_x\nonumber\\
\phi^\eps_y(\tau,x,y)&=&\psi_y(q_*(\tau)+\eps^2\wt{q}(\tau/\eps),x,y)+\eps^2 Y_y\nonumber\\
\phi^\eps_\tau(\tau,x,y)&=&(\dot{q}_*^\mu+\eps\wt{q}'(\tau/\eps))\psi_\mu(q_*(\tau)+\eps^2\wt{q}(\tau/\eps),x,y)+\eps Y_t
\eea
and $|\wt{q}'(t)|$, $\|Y_x(t)\|_{C^0}$, $\|Y_y(t)\|_{C^0}$, $\|Y_t(t)\|_{C^0}$ are bounded uniformly in $t$
and $\eps$ (again using Proposition \ref{sobolev}), so $\phi^\eps_x,\phi^\eps_y,\phi^\eps_\tau$ converge
uniformly on $[0,\tau_*]\times\Sigma$ to $\psi_{*x},\psi_{*y},\psi_{*\tau}$. Hence, $\phi^\eps$ converges
to $\psi_*$ in $C^1$.
\end{proof}


\section*{Appendix: Analytic properties of the nonlinear terms}\news

\renewcommand{\thethm}{A.\arabic{thm}}
\renewcommand{\theequation}{A.\arabic{equation}}

The proof of the local existence theorem \ref{locex} makes fundamental use of certain basic
analytic properties (Propositions \ref{fgprops} and \ref{fgdiffprop}) of the right hand sides $f,g$ of the evolution system (\ref{qf}),(\ref{Yg}). 
These properties are established by a long chain
of elementary arguments which we sketch in this appendix. We regard $\eps$ as a fixed parameter in $(0,1)$ and choose a local
parametrization $\psi:U\times\Sigma\ra S^2$ of $\M_n$ as given by Proposition \ref{psiprop}
and a compact convex neighbourhood $K\subset U$. We begin by showing that the associated maps $\Psi_k:U\ra H^k$,
$q\mapsto \psi(q,\cdot)$ are smooth
for all $k\in\N$.

\begin{lemma}\label{fsmooth} Let $k\in\N$ and $f:U\times\Sigma\ra\R$ be smooth. Then $F:U\ra\hh^k$, $F(q)=f(q,\cdot)$, is smooth.
\end{lemma}

\begin{proof} As usual, we will denote partial derivatives with respect to $q^\mu$ by a subscript $\mu$.
It suffices to show that $F$ is everywhere differentiable, since all partial derivatives
$f_{\mu_1\mu_2\cdots\mu_r}$
are, like $f$, smooth maps $U\times\Sigma\ra\R$. By the mean value theorem there exist
$t_1,\ldots, t_r\in[0,1]$ such that
\bea
\|F(q+p)-F(q)-f_\mu(q,\cdot)p^\mu\|_k^2&\leq&
|p|^2\sum_{\mu=1}^{4n}\int_\Sigma\big\{(f_\mu(q+t_1p)-f_\mu(q))^2\nonumber\\
&&+(f_{\mu x}(q+t_2p)-f_{\mu x}(q))^2+\cdots\nonumber\\
&&\cdots+
(f_{\mu y\cdots y}(q+t_rp)-f_{\mu y\cdots y}(q))^2\big\}
\eea
since $f$ and all its partial derivatives are $C^1$ functions of $q$. But $f_\mu,\ldots,f_{\mu y\cdots y}$ are uniformly
continuous on $B_\delta(q)\times\Sigma$, for $\delta>0$ sufficiently small, so
\beq
\lim_{p\ra 0}\frac1{|p|}\|F(q+p)-F(q)-\frac{\cd f}{\cd q^\mu}\bigg|_q p^\mu\|_k=0
\eeq
as was to be proved.
\end{proof}

Now, by the definition of $H^k$, $\psi:U\ra H^k$ is differentiable if and only if $\psi_i:U\ra\hh^k$ is differentiable for
$i=1,2,3$, so we immediately obtain:

\begin{corollary}\label{Psismooth} $\Psi_k:U\ra H^k$, $q\mapsto \psi(q,\cdot)$, is smooth for all $k\in\N$.
\end{corollary}

The error terms are $j'(q,\dot{q},Y,Y_t)$ and
$h(q,\dot{q},Y,Y_t)$ where
\bea
j'(q,p,Y,Z)&=&2p^\mu(\psi_\mu\cdot Z)\psi+\eps(|Z|^2-|Y_x|^2-|Y_y|^2)\psi
+\eps p^\mu p^\nu(\psi_\mu\cdot\psi_\nu)Y\nonumber\\
&&-2\eps(\psi_x\cdot Y_x+\psi_y\cdot Y_y)Y
+\eps^2\{|Y|^2\Delta\psi+2(Y\cdot Y_x)\psi_x+2(Y\cdot Y_y)\psi_y\}\nonumber \\
&&+2\eps^2 p^\mu(\psi_\mu\cdot Z)Y+\eps^3(|Z|^2-|Y_x|^2-|Y_y|^2)Y
\nonumber\\
h(q,p,Y,Z)^\mu&=&\gamma^{\mu\nu}\{\ip{Z,\psi_{\nu\lambda}}p^\lambda+\eps\ip{Y,\psi_{\lambda\nu\rho}}p^\lambda p^\rho
+\eps\ip{\psi_\lambda\cdot\psi_\rho Y,\psi_\nu}p^\lambda p^\rho\nonumber\\
&&-2\eps\ip{(\psi_x\cdot Y_x+\psi_y\cdot Y_y)Y,\psi_\nu}
+2\eps^2\ip{(\psi_\lambda\cdot Z)Y,\psi_\nu}p^\lambda\nonumber\\
&&+\eps^3\ip{(|Z|^2-|Y_x|^2-|Y_y|^2)Y,\psi_\nu}\}.
\eea
We will also need to consider the quantities
\bea
{M^{\mu}}_{\nu}(q,Y)&=&{\delta^\mu}_{\nu}-\eps^2\gamma^{\mu\lambda}(q)\ip{Y,\psi_{\lambda\nu}}\\
A(q,Y)&=& -(|\psi_x|^2+|\psi_y|^2)Y-2(\psi_x\cdot Y_x+\psi_y\cdot Y_y)\psi-2(\psi\cdot Y)\Delta\psi\nonumber\\
&&-2(\psi\cdot Y)_x\psi_x-2(\psi\cdot Y)_y\psi_y.
\eea
Let $B$ denote the Banach space $\R^{4n}\times\R^{4n}\times H^3\times H^2$
with norm 
\beq
\|(q,p,Y,Z)\|_B=\max\{|q|,|p|,\|Y\|_3,\|Z\|_2\},
\eeq
 $B_U\subset B$ denote the open set on which $q\in U$ and for each 
$\Gamma\geq 0$, $B_\Gamma\subset B_U$ denote the closed convex subset on which $q\in K$ and $\|(0,p,Y,Z)\|_B\leq\Gamma$.

\begin{prop}\label{lautob}
 The quantities defined above are smooth maps $j':B_U\ra H^2$, $h:B_U\ra\R^{4n}$, ${M^{\mu}}_{\nu}:B_U\ra\R$,
$A:B_U\ra H^2$
\end{prop}

\begin{proof} That $j', h$ define maps $B_U\ra H^2$ and $B_U\ra\R^{4n}$ follows immediately from
the algebra property of $\hh^2$ (Proposition \ref{algebra}).  Now $j'$ is a linear combination of terms
formed by composing the maps
\bea
\Psi_k:U\ra H^k,&\quad& q\mapsto \psi(q,\cdot)\nonumber\\
\d\Psi_k:U\times\R^{4n}\ra H^k,&\quad& (q,p)\mapsto p^\mu\psi_\mu(q,\cdot)\nonumber\\
\eea
which are smooth by Corollary \ref{Psismooth}, and the
manifestly smooth maps
\bea
H^k\ra\hh^k,&\quad& Y\mapsto Y_i\nonumber\\
\hh^k\ra H^k,&\quad& f\mapsto fe_i\nonumber\\
H^k\ra H^{k-1},&\quad& Y\mapsto Y,\, Y\mapsto Y_x,\, Y\mapsto Y_y\nonumber\\
\hh^2\times\hh^2\ra\hh^2,&\quad& (f,g)\ra fg,\label{likn}
\eea
where $e_1=(1,0,0)$, $e_2=(0,1,0)$ and $e_3=(0,0,1)$.
Hence $j'$ is smooth.

The map $h$ is handled similarly, after noting that the inverse metric coefficients $\gamma^{\mu\nu}$ are smooth $U\ra \R$, the
higher derivatives
\bea
\d^2\Psi_k:U\times\R^{4n}\times\R^{4n}\ra H^k,&\quad& (q,p_1,p_2)\mapsto p_1^\mu p_2^\nu\psi_{\mu\nu}(q,\cdot)\nonumber\\
\d^3\Psi_k:U\times\R^{4n}\times\R^{4n}\times\R^{4n}\ra H^k,&\quad& (q,p_1,p_2,p_3)\mapsto p_1^\mu p_2^\nu p_3^\lambda\psi_{\mu\nu\lambda}(q,\cdot)
\eea
are smooth by Lemma \ref{Psismooth} and, in addition to the maps in (\ref{likn}) above, the map 
\beq
H^0\times H^0\ra \R,\qquad (Y,Z)\mapsto \ip{Y,Z}
\eeq
is manifestly smooth.

That ${M^\mu}_{\nu}$ and $A$ define smooth maps on $B_U$ is clear, since their $q$-dependence is smooth, and they depend linearly on $Y$ (and are 
independent of $p$ and $Z$).
\end{proof}

We can now assemble these pieces to show that $f$ and $g$ are smooth functions on $X_\Gamma$, the space defined in section \ref{sec:locex}. 
To do so, we note that
$f=\wh{f}\circ\iota$ and $g=\wh{g}\circ\iota$ where $\iota:X_\Gamma\ra B_{8c(\Sigma)\Gamma}$ is the linear isometry
$$
\iota:(q,p,Y,Z)\ra (q,\eps^{-1}p,Y,Z)
$$
and $\wh{f}:B_{\Gamma'}\ra \R^{4n}$,$\wh{g}:B_{\Gamma'}\ra H^2$ are
\bea
\wh{f}(q,p,Y,Z)&=&M^{-1}(q,Y)(-\eps^2 G(q,p,p)+\eps h(q,p,Y,Z))\\
\wh{g}(q,p,Y,Z)&=&-A(q,Y)-\psi_\mu \wh{f}^\mu(q,p,Y,Z)-\psi_{\mu\nu}p^\mu p^\nu+\eps j'(q,p,Y,Z),
\eea
and $G$ is defined in (\ref{lato}). The point is that $f,g$ are defined as functions of $q_t$ (and $q,Y,Y_t$) on a space ($X_\Gamma$)
with $\eps$-dependent norm, but it is more covenient here to think of them as functions of $\dot{q}=\eps^{-1}q_t$, on a space ($B_{\Gamma'}$, 
$\Gamma'=8c(\Sigma)\Gamma$) with fixed norm. Since $\iota$ is a linear isometry, $f,g$ are smooth, bounded, Lipschitz, etc.\ if and only if $\wh{f},
\wh{g}$ are. 

\begin{proposition} There exists $\eps_*=O(1/\sqrt{\Gamma})$ such that for all $\eps\in(0,\eps_*)$, $\wh{f}:B_\Gamma\ra\R^{4n}$ and
$\wh{g}:B_\Gamma\ra H^2$ are smooth.
\end{proposition}

\begin{proof} Since $\gamma^{\mu\lambda}(q)$, $\psi_{\lambda\nu}(q)$ are smooth and $K$ is compact, there exists $\eps_*>0$ such that the
matrix $M(q,Y)$ is uniformly invertible on $B_\Gamma$ for all $\eps\in(0,\eps_*)$. The components of $M^{-1}$ are rational in ${M^\mu}_{\nu}$
with denominator $\det M$, which is bounded away from $0$ on $B_\Gamma$. Hence, $M^{-1}$ is smooth on $B_\Gamma$, and the proposition follows 
immediately from the Leibniz rule and Proposition \ref{lautob}.
\end{proof}

Since $\wh{f}$ and $\wh{g}$ are smooth, they are certainly continuously differentiable. To complete the proof of Proposition \ref{fgprops}, it
remains to show that their differentials are bounded.
\begin{prop}\label{sa-jame}
For all $\eps\in(0,\eps_*)$, the
derivatives $\d\wh{g}:B_\Gamma\ra{\cal L}(B,H^2)$ and $\d\wh{f}: B_\Gamma\ra{\cal L}(B,\R^{4n})$ are bounded, independent of $\eps$.
\end{prop}

\begin{proof} This is established by estimating the operator norm of the derivatives termwise. For example, the first term of $j'$,
\beq
J(q,p,Y,Z)=2 p^\mu(\psi_\mu\cdot Z)\psi
\eeq
has derivative
\beq
dJ_{(q,p,Y,Z)}:(\hat{q},\hat{p},\hat{Y},\hat{Z})\mapsto 2\hat{q}^\nu p^\mu[(\psi_{\nu\mu}\cdot Z)\psi+(\psi_\mu\cdot Z)\psi_\nu]
+2\hat{p}^\mu(\psi_\mu\cdot Z)\psi
+2 p^\mu(\psi_\mu\cdot\hat{Z})\psi,
\eeq
and so, by the algebra property of $\hh^2$, for all $(q,p,Y,Z)\in B_\Gamma$,
\bea
\|dJ_{(q,p,Y,Z)}(\hat{q},\hat{p},\hat{Y},\hat{Z})\|_2&\leq& C(K)\{|\hat{q}||p|\|Z\|_2+|\hat{p}|\|Z\|_2+\|Z\|_2|p|\}\nonumber\\
&\leq& C(K)\{2\Gamma+\Gamma^2\}\|(\hat{q},\hat{p},\hat{Y},\hat{Z})\|_B
\eea
where $C(K)>0$ is a constant depending only on $K$. Hence, for all $(q,p,Y,Z)\in B_\Gamma$,
\beq
\|d J_{(p,q,Y,Z)}\|_{{\cal L}(B,H^2)}\leq C(K)\{2\Gamma+\Gamma^2\}.
\eeq
The other terms of $j', h$ and $A$ are handled similarly. To bound $\d M^{-1}$ we bound $\d M$ and appeal to the Leibniz rule and
uniform invertibility of $M$. Boundedness of the differentials of functions depending only on $q$ and $p$ is immediate by compactness and finiteness
of dimension.
\end{proof}

Finally, we turn to Proposition \ref{fgdiffprop}, which is equivalent to:

\begin{prop} 
The differentials of the maps $\wh{f}:B_\Gamma\ra\R^{4n}$ 
$\wh{g}:B_\Gamma\ra H^2$ extend to maps $\d\wh{f}^{ext}:B_\Gamma\ra{\cal L}(\R^{4n}\times \R^{4n}\times H^1\times
L^2,\R^{4n})$ and
$\d\wh{g}^{ext}g:B_\Gamma\ra{\cal L}(\R^{4n}\times \R^{4n}\times H^1\times
L^2,L^2)$ bounded by $\Lambda_f$, $\Lambda_g$ respectively.
\end{prop}

\begin{proof} This follows from explicit termwise computation.
For example, the last term of $h$ is
$$
m_\mu(q,Y,Z)=\ip{(|Z|^2-|Y_x|^2-|Y_y|^2)Y,\psi_\mu}
$$
whose differential at $(q,Y,Z)$ is the linear map $\R^{4n}\times H^3\times H^2\ra \R$
\ben
\d m_\mu:(\hat{q},\hat{Y},\hat{Z})&\mapsto&\ip{(|Z|^2-|Y_x|^2-|Y_y|^2)Y,\psi_{\mu\nu}}\hat{q}^\nu+
\ip{(|Z|^2-|Y_x|^2-|Y_y|^2)\psi_\mu,\hat{Y}}\\
&&-2\ip{(Y\cdot\psi_\mu)Y_x,\hat{Y}_x}
-2\ip{(Y\cdot\psi_\mu)Y_y,\hat{Y}_y}
+2\ip{(Y\cdot\psi_\mu)Z,\hat{Z}}
\een
which clearly extends to a bounded linear map $\R^{4n}\times H^1\times L^2\ra\R$. 
\end{proof}

\subsection*{Acknowledgements} This work was financially supported by EPSRC. It 
began as a collaboration with Mark Haskins, and
I gratefully acknowledge many hours of valuable discussion of the problem with him. I also
benefited from conversations with David Stuart and Matthew Daws. The work was 
largely completed
while I was on sabbatical leave at the 
Isaac Newton Institute for Mathematical Sciences in Cambridge.

\end{document}